\documentclass[12pt,twoside]{article}          %pour un article
\usepackage[latin1]{inputenc}
\usepackage[T1]{fontenc}
\usepackage[english]{babel}
\usepackage{amsmath, amssymb, amsthm}            
\usepackage{amstext, amsfonts, a4}
\usepackage{tikz}

 \usepackage[colorlinks=true]{hyperref}
\hypersetup{urlcolor=blue, citecolor=blue}

\usepackage{color,transparent}
\usepackage{graphics,graphicx,subfigure}

\usepackage{comment}
\usepackage{float} 
\theoremstyle{plain}
	\newtheorem{Theo}{Theorem}[section] 
	\newtheorem{Prop}[Theo]{Proposition}

	\newtheorem{Rema}[Theo]{Remark}

\theoremstyle{definition}

\theoremstyle{remark}

\def\NN{{\mathbb N}}    %naturels
\def\RR{{\mathbb R}}    %réels
\def\PP{{\mathbb P}}     %eproba
\def\EE{{\mathbb E}}    %espérance

\def\dt{{\mathrm d}}
\def\tL{{\mathtt L}}
\def\tP{{\mathtt P}}

\def\expo{{\mathrm e}}

\def\Hun{{(\mathcal{H}_1)}}
\def\Hdeux{{(\mathcal{H}_2)}}
\def\Htrois{{(\mathcal{H}_3)}}
\def\Hquatre{{(\mathcal{H}_4)}}

\def\normx{{\lvert x \lvert}}
\def\normv{{\lvert v \lvert}}
\def\xscalv{{\frac{x\cdot v}{\lvert x \lvert}}}

\makeatletter
\def\BState{\State\hskip-\ALG@thistlm}
\makeatother

\begin{document}
\title{Long-time behaviour of generalized Zig-Zag process}

\author{Ninon F\'{e}tique \footnotemark[1]}
\date{ }

\footnotetext[1]{Laboratoire de Math\'ematiques et Physique Th\'eorique (UMR CNRS 7350), F\'ed\' eration Denis Poisson (FR CNRS
2964), Universit\'e Fran\c cois-Rabelais, Parc de Grandmont, 37200 Tours, France. Email: ninon.fetique@lmpt.univ-tours.fr}

\maketitle

\begin{abstract}
We study the long-time behaviour of a class of piecewise-deterministic Markov processes which are an extension of some recent works. These $d$-dimensional processes, $d\geq 1$, can especially be used to model the motion of a bacterium in presence of a chemo-attractant, with parameters depending both on the position and the velocity of the bacterium. Using the method of Meyn and Tweedie (\cite{MT93a,MT93b}), we show that under some good assumptions on the parameters of the model, such a process converges exponentially fast towards its invariant measure. We also establish the existence of exponential moments of the invariant measure using results on semi-regenerative processes. The one-dimensional case is studied separately since complementary results can be obtained in that particular case.

\end{abstract}

\begin{flushleft}
\textbf{MSC Classification 2010:} 60F17, 60J05, 60J25, 60J75. 
\end{flushleft}

\tableofcontents

\bigskip

\section{Introduction}\label{Section introduction}   

Piecewise-deterministic Markov processes (PDMPs) have been introduced by Davis (\cite{Dav}) to distinguish these particular processes from diffusive processes. They are the subject of much current work in various domains, since they are a simple alternative to diffusions to model stochastic systems (see \cite{ABGKZ} for an overview of recent results on PDMPs). In this paper, we study a PDMP that comes from biological modeling for the motion of flagellated bacteria which are in presence of a chemo-attractant. The motion of such a bacterium has been described as run-and-tumble, which means that the bacterium alternates sequences of linear runs with periods of random reorientation (tumbling). The tumble-periods being typically much shorter than the run-periods, we can suppose them to be instantaneous. Moreover, the presence of a chemo-attractant in the environment of the bacterium influences the rate at which the bacterium changes its direction, and also the new direction it takes (see \cite{ODA} for more details on the model).\\
More precisely, we consider the PDMP $((X_t,V_t))_{t\geq 0}$ with values in $E=\RR^d\times \mathcal{B}(1) $, where $\mathcal{B}(1)=\left\{ v\in\RR^d, \normv \leq 1 \right\}$ is the Euclidian ball of radius $1$, with infinitesimal generator given by, for $h$ in the domain of $L$ (see \cite{Dav} for a definition of the domain of the infinitesimal generator):
\begin{equation}\label{generateur}
Lh(x,v)=v\cdot \nabla_x h(x,v)+\lambda(x,v)\int_{\mathcal{B}(1)} \left( h(x,v')-h(x,v)\right)Q(x,v,\mathrm{d}v'),
\end{equation}
where $Q(x,v,\cdot)$ is a probability kernel on $\mathcal{B}(1)$. We call this process "generalized Zig-Zag process" since it is a generalization of the Zig-Zag process studied in \cite{BR} and \cite{BFR}, in the sense that we do not any more have a velocity with values in $\{-1,1\}^d$, but in $\mathcal{B}(1)$.\\
In our model, $X_t$ represents the position of a bacterium at instant $t$, and $V_t$ its velocity. The form of the generator indicates that the first component $X$ is continuous and evolves according to $\frac{\dt X_t}{\dt t}=V_t$, whereas $V$ is constant during a random time, and jumps according to the kernel $Q$ at rate $\lambda(x,v)$ when $(X_t,V_t)=(x,v)$. The fact that the motion of the bacterium is biaised by the presence of a chemo-attractant will be taken into account in the assumptions we will make further on the jump rate $\lambda$ and the velocity kernel $Q$.\\
In this paper, we are interested in the long-time behaviour of the process driven by \eqref{generateur} under some good assumptions.

\bigskip
Many special cases of this process, driven by \eqref{generateur}, have already been studied in different ways and under different assumptions. Let first mention some works on the process in dimension $1$ with a modeling point of view: in \cite{FGM10}, Fontbona, Gu\'{e}rin and Malrieu have shown the exponential ergodicity of the process with a jump rate equal to $a\mathbf{1}_{xv\leq 0}+b\mathbf{1}_{xv>0}$ with $b> a >0$, and the velocity taking its values in $\{-1,+1\}$. For this, they give an exact description of the excursions of the process away from the origin and give an explicit construction of a coalescent coupling for both velocity and position. In \cite{FGM15} and \cite{BR}, the previous result has been extended by considering a more general jump rate, depending on the position and the velocity of the bacterium. Calvez, Raoul and Schmeiser have shown in \cite{CRS} by an analytical method the exponential ergodicity of the process driven by \eqref{generateur} in the particular case where the kernel $Q$ is the uniform kernel on $[-1,1]$, and under similar assumptions to the ones introduced here (see Section \ref{Section The particular case of dimension $1$}).\\
Furthermore, there exist also results for this process in high dimension. We can for instance cite \cite{BFR}, in which Bierkens, Fearnhead and Roberts study the Zig-Zag process, that is the process with values in $\RR^d\times \{-1,+1\}^d $. They prove its ergodicity in the case where it can be seen as a product of independent one-dimensional Zig-Zag processes. Then, Bierkens, Roberts and Zitt generalize these results in \cite{BRZ}, not considering only the case of a product of one-dimensional Zig-Zag process. In \cite{BCDV}, \cite{DBCD} and \cite{Mon}, the authors are interested in the ergodicity of the bouncy particle sampler, with values in $\RR^d\times\RR^d$ or $\RR^d\times\mathcal{S}^{d-1}$, where $\mathcal{S}^{d-1}$ is the unit sphere of $\RR^d$. This PDMP is a particular case of the process driven by \eqref{generateur}: for instance in \cite{Mon}, the jump rate is given by $\lambda(x,v)=\left( v\cdot \nabla_x U(x) \right)_+$, where $U$ is a potential, and at each jump, the velocity is reflected according to optical laws on the level set of $U$ it has reached. Recently, a new model have been studied: in \cite{RW}, Robert and Wu introduce the coordinate sample, which is a variant of the Zig-Zag process, since the velocity does not live in $\{-1,+1\}^d$ but in $\{e_i,1\leq i \leq d\}$ the canonical base of $\RR^d$. For this process, only one component of the position evolves between the jumps. The authors show in the paper the exponential ergodicity of this process under some conditions.\\
Finally, let us mention that the study of this kind of processes has an interest not only for biological modelling, but also for simulating a target distribution. In fact, almost all of the recent study in dimension $d$ do this for the sampling. Let us refer to \cite{BFR,BR,BRZ,BCDV,DBCD,DGM,Mon,RW}, where the authors want to sample from a distribution with a density proportional to $\expo^{ -U}$, where $U$ is a potential on $\RR^d$. For a jump rate $\lambda$ and a jump kernel $Q$ well chosen (with respect to $U$), the PDMP converges towards the targeted distribution. The estimation of the speed of convergence to equilibrium of these processes gives then information on the efficiency of the corresponding algorithms to sample from the target distribution. Comparisions of the efficiency of the different samplers are done in \cite{ADNR} and \cite{RW} for instance. An interest of considering PDMPs to catch a distribution is the irreversibility of PDMPs. Indeed, while many Markov chain Monte Carlo (MCMC) methods rely on realisations from a discrete reversible ergodic Markov chain, it has been observed that non-reversibility often implies favourable convergence properties (see for instance \cite{HHS,LNP}). 
Moreover, PDMPs have the advantage to be easy to sample, and can even in some cases being simulated without discretisation error.

 \subsection*{Framework}
 
 Let us now introduce the framework of the paper. Denoting by $x\cdot v$ the scalar product of $x\in \RR^d$ and $v\in \RR^d$, and $\lvert x \lvert$ the Euclidian norm of the vector $x$, the assumptions made on the model are the following:
 
\begin{itemize}

\item[$\Hun$ :] There exists $\lambda_{\min}>0$ such that for all $(x,v)\in E$, $\lambda(x,v)\geq \lambda_{\min}$;
\item[$\Hdeux$ :] The quantity $\lambda_{\max}=\sup_{\{(x,v)\in E ~:~ x\cdot v \leq 0\}}\lambda(x,v)$ is finite;
\item[$\Htrois$ :] There exists $p>0,\theta_0,\bar{\theta}\in (0,1]$ such that for all $(x,v)\in E$ satisfying $\frac{x\cdot v}{\normx}>-\bar{\theta}$, \[\int_{\{ v'\in \mathcal{V} ~:~ \frac{x\cdot v'}{\normx}\leq -\theta_0\}} Q(x,v,\dt v')\geq p;\]
\item[$\Hquatre$ :] There exists $\theta_*\in \left[0,(p\theta_0)^2\frac{\lambda_{\min}}{\lambda_{\max}}\right)$, $\beta>\frac{1}{(p\theta_0)^2}$ and $\Delta>0$ such that \[\inf_{\left\{ \xscalv \geq \theta_*, \normx \geq \Delta \right\}}\lambda(x,v) \geq \beta\lambda_{\max}.\]

\end{itemize}

Assumptions $\Hun$ and $\Hdeux$ mean that the jump rate is uniformly bounded by below, and that it is bounded from above when the bacterium is moving towards $0$, where the chemo-attractant is assumed to be. The existence of the lower bound $\lambda_{\min}$ is here to ensure the irreducibility of the process. We refer to Section 1.3 of \cite{BRZ} for illustrations of the difficulties if the switching rate $\lambda$ can be zero.\\
Assumptions $\Htrois$ and $\Hquatre$ reflect the attraction of the bacterium to the origin. Indeed, in Assumption $\Htrois$, we suppose that if the bacterium does not "enough" go towards the origin, when a jump happens, it has a chance to go towards it. Moreover, in Assumption $\Hquatre$, we assume a kind of monotonicity of the jump rate. Roughly speaking, we suppose that when the bacterium is far from the origin, and goes in a too bad direction, its jump rate is strictly bigger than $\lambda_{\max}$, which is the maximum of the jump rate when the bacterium is coming closer the origin.\\
Assumptions $\Hun,\Hdeux$ and $\Hquatre$ appear in the works \cite{BR,FGM10,FGM15}. However, in these papers, there is no assumption equivalent to $\Htrois$. But it is in fact normal: the Zig-Zag process has a velocity in $\{-1,1\}$, and thus, as soon as the bacterium goes in a bad direction, the jump makes it go towards the origin, and the assumption is in fact satisfied.\\
All of these assumptions seem to be reasonable for the modelling of the motion of bacteria as described above.

\begin{Rema}
We only consider the case of the speed in the ball unit ball of $\RR^d$, but the results can easily be adapted if the speed lies in a compact set $K$ of $\RR^d$.\\
Moreover, although we could hope to obtain similar results with an unbounded velocity (with probably other assumptions on the jump rate), we do not deal with this case since it is not really relevant from a modelling point of view.
\end{Rema}

Under these assumptions, we can show the exponential ergodicity (see Section \ref{Section Preliminaries} for the definition) of the generalized Zig-zag process $(X,V)$. 

\begin{Theo}\label{Theoprincipal}
Let $(X_t,V_t)_{t\geq 0}$ be a PDMP on $E=\RR^d\times \mathcal{B}(1)$ with infinitesimal generator given by \eqref{generateur}.\\
If $\lambda$ and $Q$ satisfy Assumptions $\Hun$, $\Hdeux$, $\Htrois$ and $\Hquatre$, then the process is exponentially ergodic.
\end{Theo}

\begin{figure}\label{histogrammes}
\includegraphics[scale=0.25]{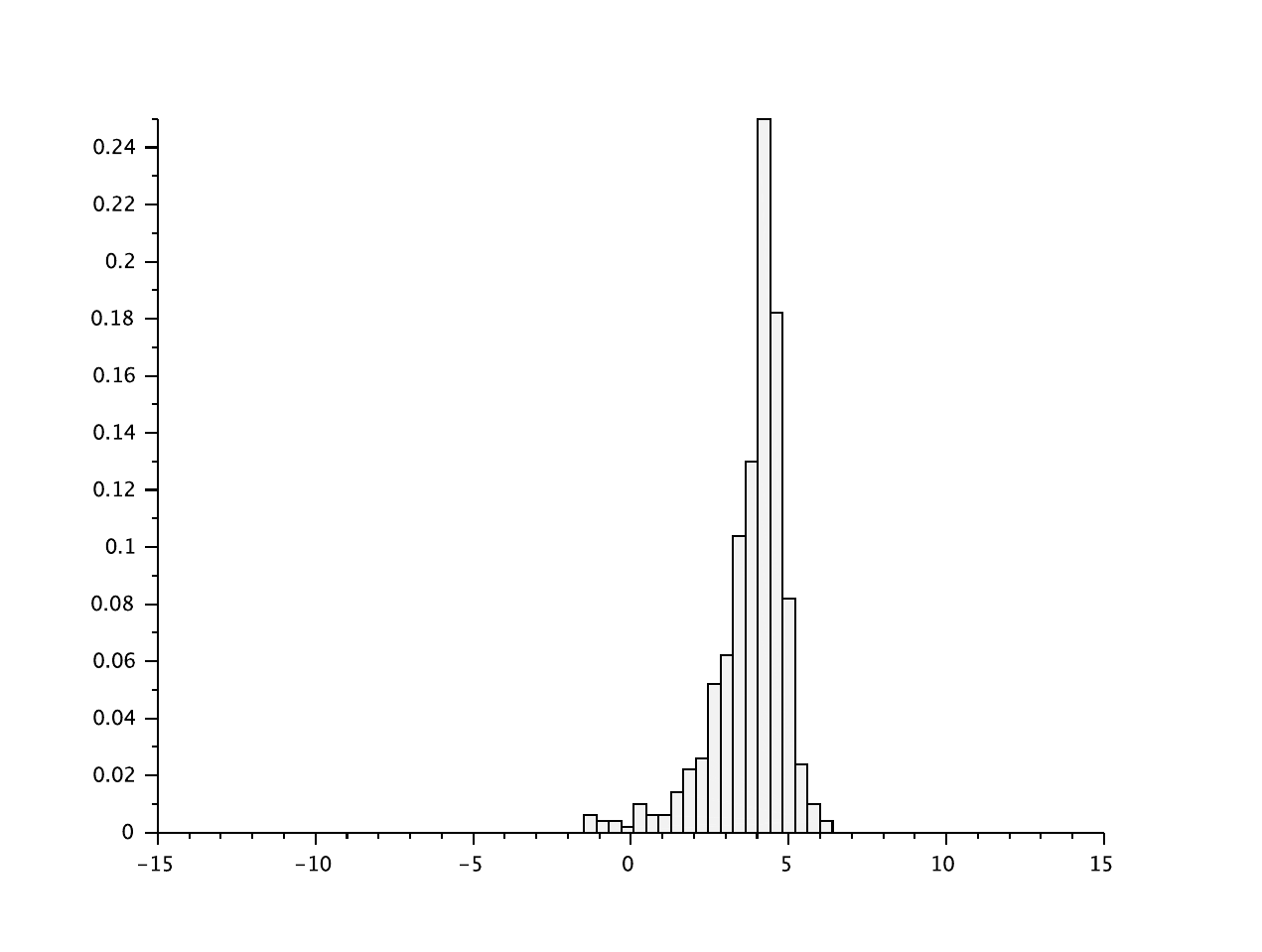} \includegraphics[scale=0.25]{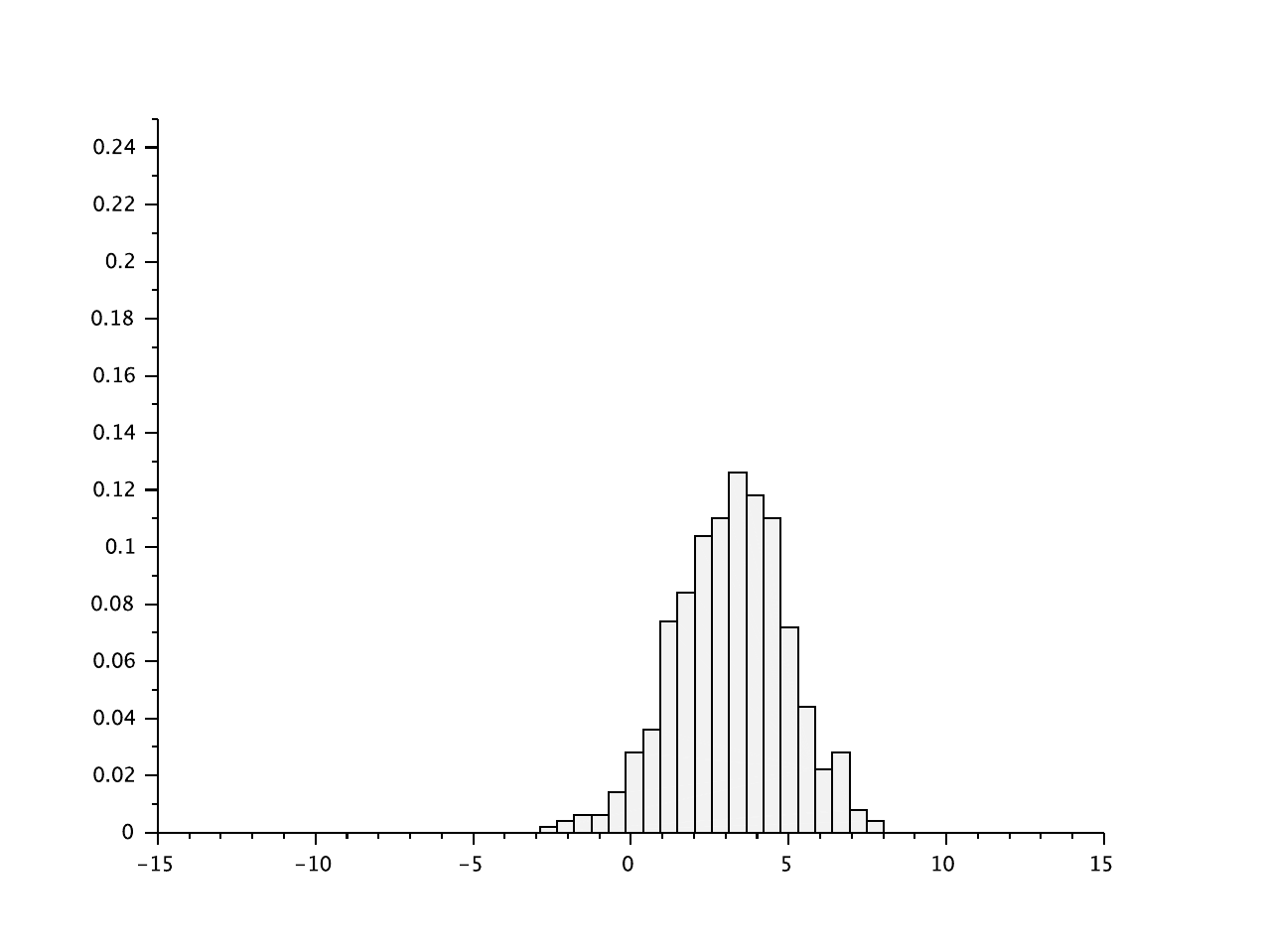} \includegraphics[scale=0.25]{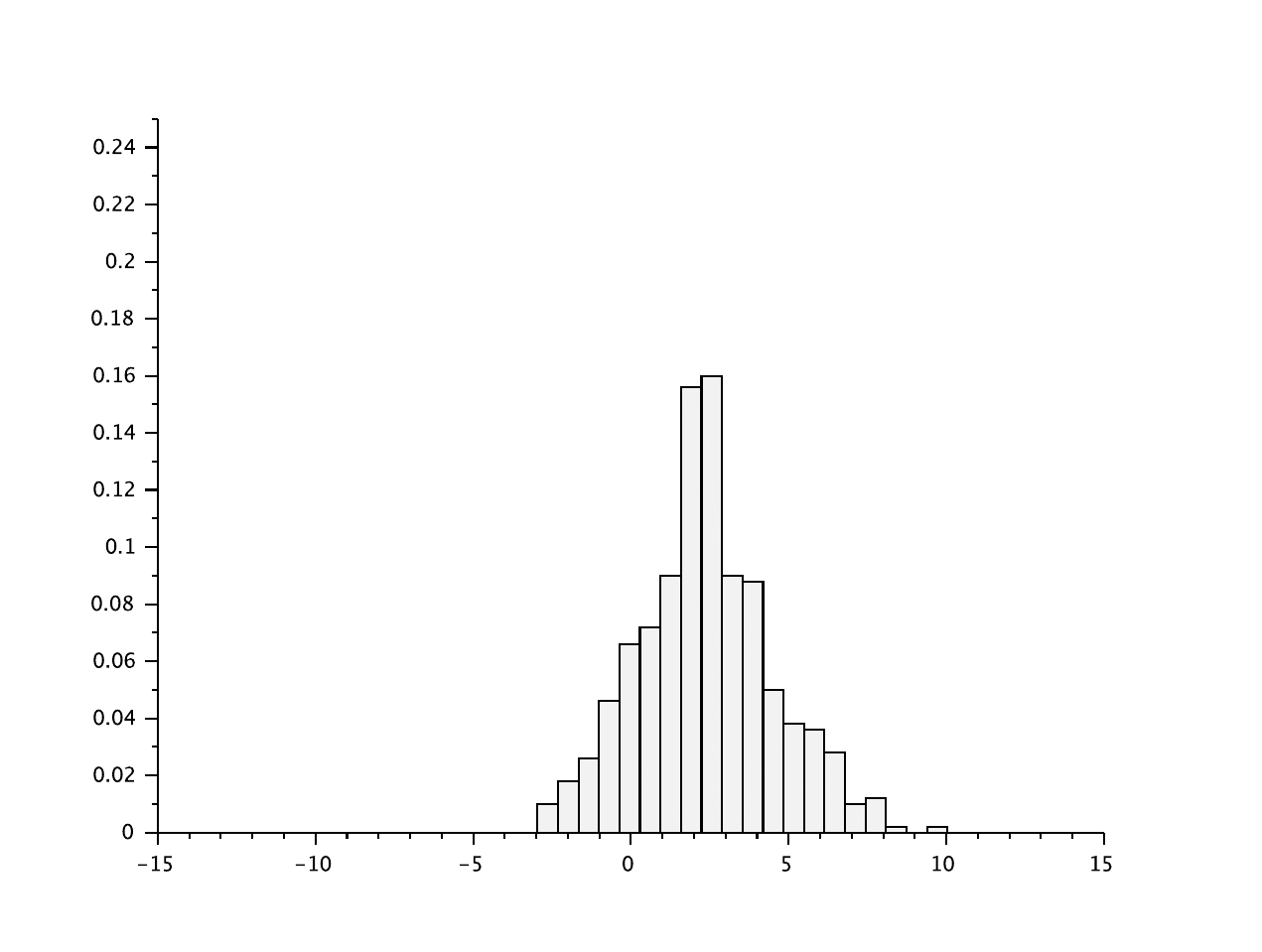} \includegraphics[scale=0.25]{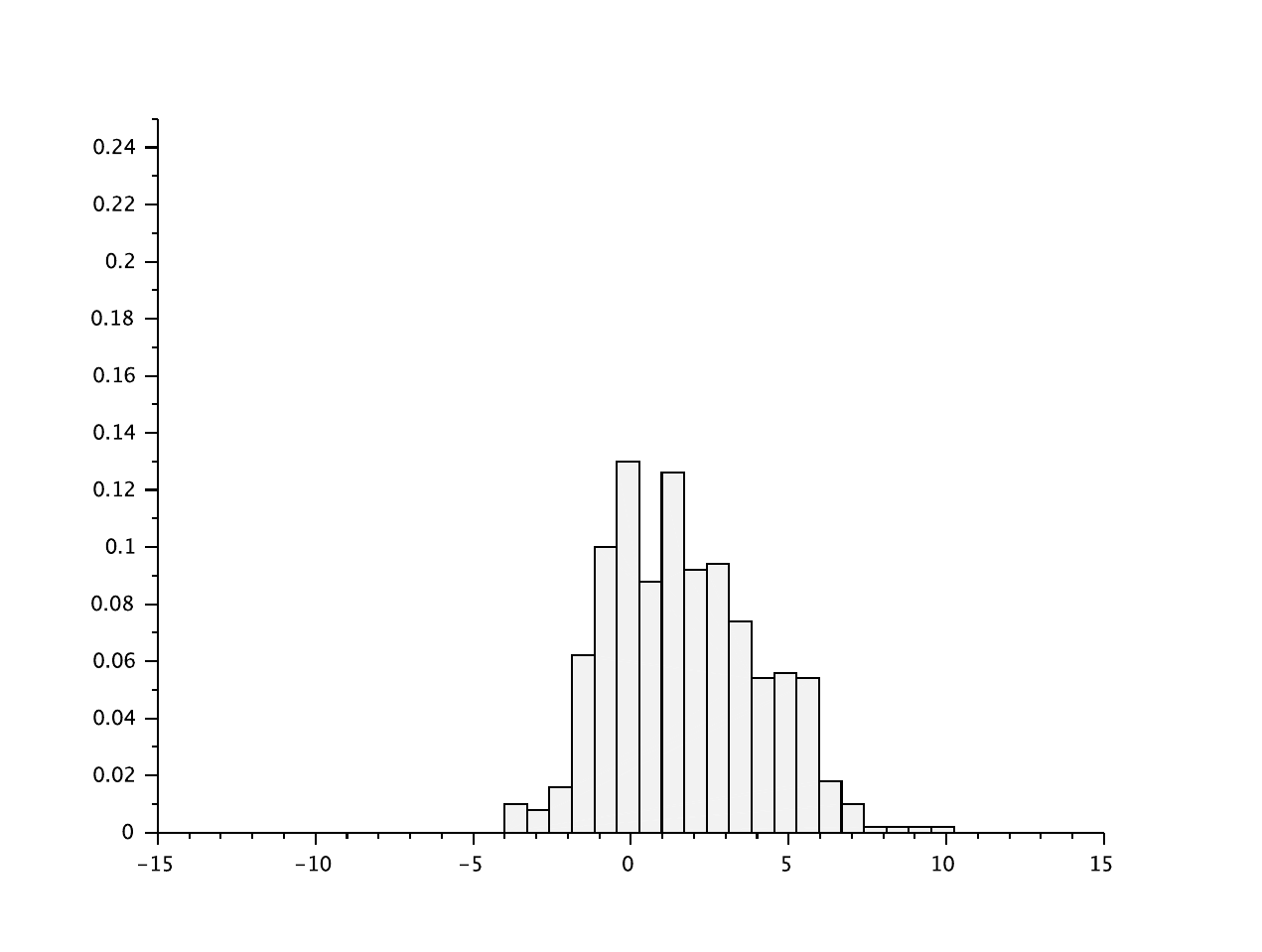} \includegraphics[scale=0.25]{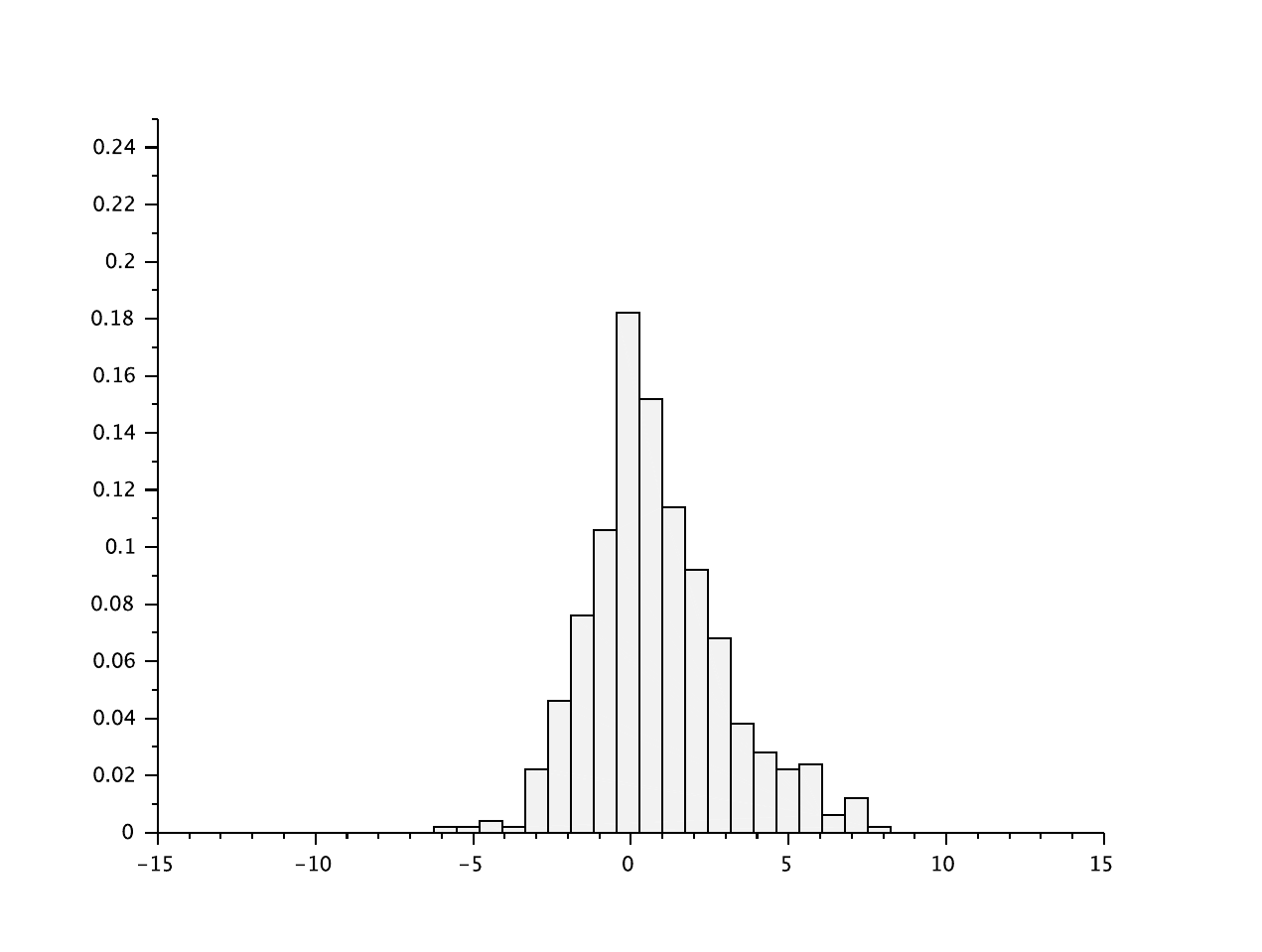} \includegraphics[scale=0.25]{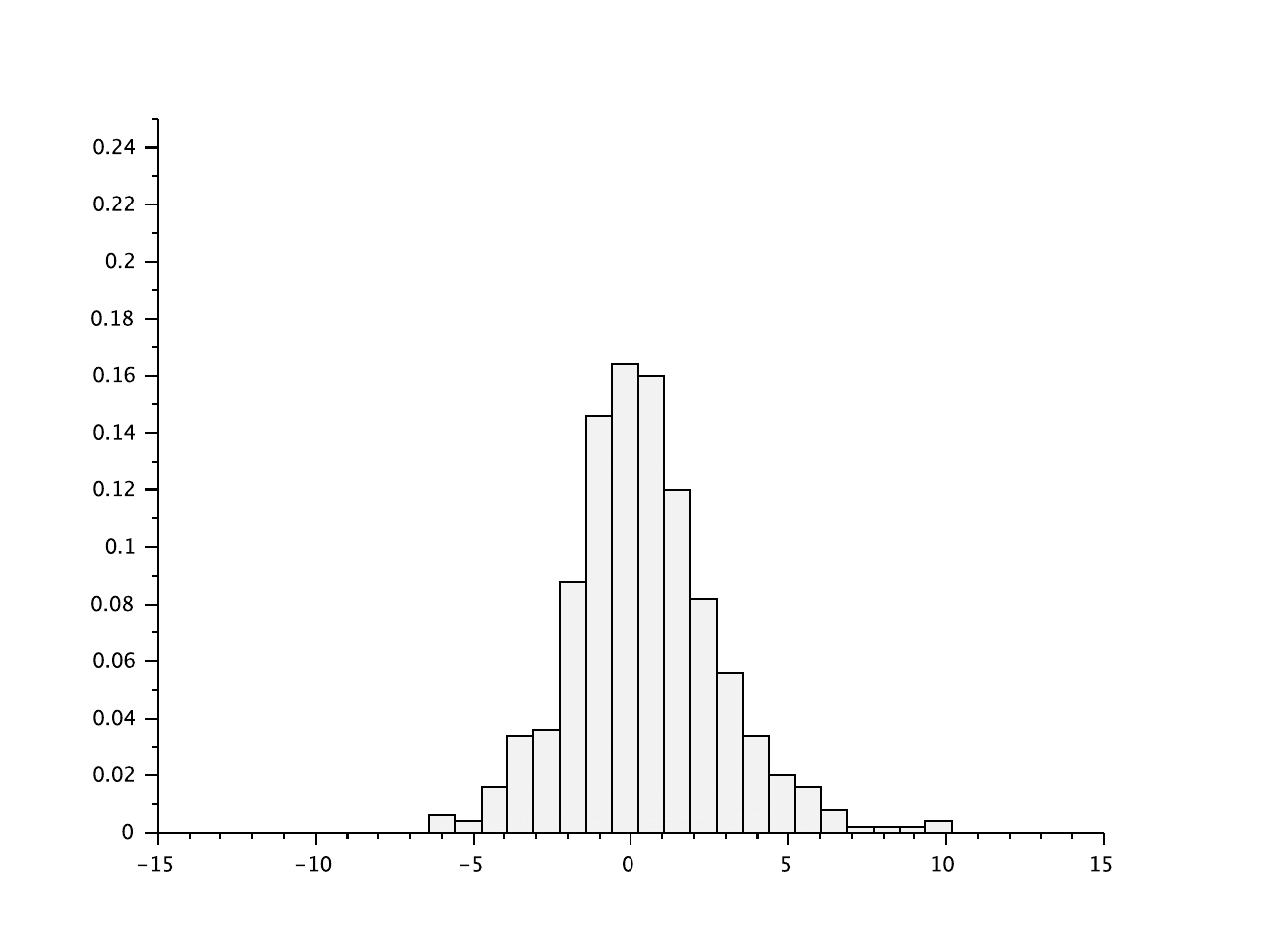}
\caption{Empirical distribution of $X_t$ starting at $(5,-1)$ for $t\in\{0.5,5,10,17,30,40\}$ with $\lambda(x,v)=\mathbf{1}_{xv<0}+2\mathbf{1}_{xv\geq 0}$ and $Q(x,v,\dt v')=\frac{1}{2}\mathbf{1}_{[-1,1]}(v')\dt v'$.}
\end{figure}

In Figure \ref{histogrammes}, we can observe the convergence of the empirical law of $(X_t)_{t\geq 0}$ in the case $\lambda(x,v)=\mathbf{1}_{xv<0}+2\mathbf{1}_{xv\geq 0}$ and $Q(x,v,\dt v')=\frac{1}{2}\mathbf{1}_{[-1,1]}(v')\dt v'$.\\

\bigskip
The interest of this convergence result in relation to those already existing is that it concerns a very general class of PDMPs.\\
The approach we will carry out in Section \ref{Section Main result} to prove Theorem \ref{Theoprincipal} is the classical method of Meyn and Tweedie (\cite{DMT,MT93a,MT93b}), by showing the existence of a Lyapunov function for the process, and the existence of petite sets. In Section \ref{Section Preliminaries}, we thus first briefly recall the generalities about ergodicity of Markov processes.\\
Then, we will study in Section \ref{Section Exponential moments for the invariant measure} the existence of exponential moments for the invariant measure in Theorem \ref{Theo exponential moments of the invariant measure}. For this purpose, we will use a convergence result on semi-regenerative processes, since our PDMP is such a process.\\
Finally, we will go back to the particular case of the dimension $1$ in Section \ref{Section The particular case of dimension $1$}, in which we study our process with a different approach than in Section \ref{Section Main result}. In particular, in this section, we will assume that the jump kernel does not depend on the position of the bacterium. Thus, only the jump rate will favour the return towards the origin, and therefore the ergodicity of the process. We will establish in Theorems \ref{Theoprincipal dim1} and \ref{Theorem exponential moments of the invariant measure dim 1} the exponential ergodicity of the process and the existence of exponential moments for its invariant probability measure. The results obtained in this section can be linked with the previous works \cite{CRS}, \cite{FGM10} and \cite{FGM15}, since they are in fact a generalization of these latter.\\
We mention at the end of the paper that the two studies applied in dimension $1$ are complementary, since one can be applied at some process whereas the other can not, and conversely.

\section{Preliminaries}\label{Section Preliminaries}

\subsection{About ergodicity}

In this paper, we study the convergence of our process with the total variation distance. Let us recall its definition. Let $\nu$ and $\overset{\sim}{\nu}$ be two probability measures on a measurable space $E$. We say that a probability measure on $E\times E$ is a coupling of $\nu$ and $\overset{\sim}{\nu}$ if its marginals are $\nu$ and $\overset{\sim}{\nu}$. Denoting by $\Gamma(\nu,\overset{\sim}{\nu})$ the set of all the couplings of $\nu$ and $\overset{\sim}{\nu}$, we say that two random variables $X$ and $\overset{\sim}{X}$ satisfy $(X,\overset{\sim}{X})\in\Gamma(\nu,\overset{\sim}{\nu})$ if $\nu$ and $\overset{\sim}{\nu}$ are the respective laws of $X$ and $\overset{\sim}{X}$. The total variation distance between these two probability measures is then defined by 
\begin{equation}
\lVert \nu - \overset{\sim}{\nu} \lVert_{TV} = \underset{(X,\overset{\sim}{X})\in\Gamma(\nu,\overset{\sim}{\nu})}{\inf}  \PP(X\neq \overset{\sim}{X}).
\end{equation}
For other definitions of the total variation distance and its properties, see for instance \cite{Lind}.
\bigskip

Let us now introduce some usefull results to study the convergence of a Markov process towards its invariant measure (see \cite{DMT,MT93a,MT93b}).\\
Let $(Y_t)_{t\geq 0}$ be a Markov process on the state space $E$, and denote by $\tP$ its semi-group and $\tL$ its infinitesimal generator. We say that the Markov process $Y$ is exponentially ergodic if there exists a probability measure $\pi$, a function $M:E\longrightarrow \RR_+$ and a constant $0<\rho<1$ such that
\begin{equation}
\lVert \tP_t(x,\cdot)-\pi \lVert_{TV} \leq M(x) \rho^t, \text{ for all } t\geq 0,
\end{equation}
where $\tP_t(x,\cdot)=\PP_x(Y_t\in\cdot)$.\\
A set $K$ is said to be petite for the process $Y$ if there exists a probability measure $\nu$ on $\RR_+$ and a nontrivial measure $\mu$ on $E$ such that, for all $x\in K$,
\begin{equation}\label{petite}
\int_0^{\infty} \tP_t(x,\cdot)\nu(\dt t) \geq \mu(\cdot).
\end{equation}
This notion is weaker than the notion of small sets: $K$ is said to be a small set for the process $Y$ if there exists $t_0>0$ and a non trivial measure $\mu$ on $E$ such that, for all $x\in K$,
\begin{equation}
\tP_{t_0}(x,\cdot) \geq \mu(\cdot).
\end{equation}
Let $K\subset E$ be a compact set, and let $H:E\longrightarrow \RR$ be a function. We say that $H$ is a Lyapunov function (associated to the set $K$) for the process $Y$ if $H(x)\geq 1$ for all $x\in E$ and if there exists some constants $\alpha>0$ and $\beta\geq 0$ such that for all $x\in E$,
\begin{equation}\label{lyapunov}
\tL H(x) \leq -\alpha H(x) + \beta \mathbf{1}_K(x).
\end{equation}
Finally, we recall that the Markov process $Y$ is called $\varphi$-irreducible if there exists a $\sigma$-finite measure $\varphi$ such that for all measurable set $A$ with $\varphi(A)>0$ we have, for all $x\in E$, $\EE_x\left[ \int_0^{\infty} \mathbf{1}_{Y_t\in A}\dt t\right] >0$. In that case, there exists a maximal irreducibility measure $\psi$ such that for any other irreducibility measure $\nu$, $\nu$ is absolutely continuous with respect to $\psi$. We then write $\mathcal{B}^+(E)=\{A\subset E \text{ measurable}, \psi(A)>0\}$.
The process $Y$ is said to be aperiodic if for some small set $C\in \mathcal{B}^+(E)$, there exists $T>0$ such that for all $t\geq T$ and all $x\in C$ we have $\tP_t(x,C)>0$.
\bigskip

We can now recall the main result we will use in Section \ref{Subsection Proof of the main result} to prove Theorem \ref{Theoprincipal}.

\begin{Theo}{(\cite{DMT})}\label{Theo Lyapunov function}
Let $Y$ be an irreducible and aperiodic Markov process. If there exists a petite set $K$ for the process $Y$ and a Lyapunov function associated to this set $K$ (i.e. \eqref{petite} and \eqref{lyapunov} hold) , then the process $Y$ is exponentially ergodic.
\end{Theo}

\subsection{Description of the process}

Let us now describe the dynamics of the process.\\
The variables $0=T_0,T_1,T_2,\ldots$ refer to the successive jumping times of the process, and for $n\geq 1$ we denote by $\tau_n$ the inter-jump time: $\tau_n=T_n-T_{n-1}.$\\
In order to make the paper easier to read, we note $V_i$ the velocity on the time interval $[T_i,T_{i+1})$, instead of $V_{T_i}$.\\
Finally, we denote by $\overline{F}((\cdot,x,v)$ the survival function of $T_1$ with initial data $(X_0,V_0)=(x,v)$: $\overline{F}(\cdot,x,v)=\PP_{x,v}(T_1>\cdot)$.\\
The time $T_1$ satisfies, when $(X_0,V_0)=(x,v)$:
\[ T_1=\inf\left\{ t\geq 0, \int_0^t \lambda(X_s,V_s) \dt s \geq E\right\} = \inf\left\{ t\geq 0, \int_0^t \lambda(x+v s,v) \dt s \geq E\right\} \]
where $E$ is an exponential variable with unit mean, because the process is deterministic between jump times. We then get:
\begin{align*}
\overline{F}(t,x,v) &= \PP_{x,v}(T_1>t)\\
&= \PP_{x,v}\left( \int_0^t \lambda(x+v s,v) \dt s \leq E\right)\\
&= \exp\left( - \int_0^t \lambda(x+v s,v)\dt s\right).
\end{align*}
The conditional distribution of the inter-jump times is then given by, for all $n\geq 0$:
\[ \PP\left( \tau_{n+1}\geq t \left| X_{T_n},V_n \right. \right)= \exp\left( - \int_0^t \lambda(X_{T_n}+V_n s,V_n)\dt s\right). \]

\section{Main result}\label{Section Main result}

In this section we prove our main result, Theorem \ref{Theoprincipal}, using Theorem \ref{Theo Lyapunov function}. We first need to find a Lyapunov function for the process.

\subsection{A Lyapunov function}\label{Subsection A Lyapunov function}

Let us introduce some constants which will appear in the definition of our Lyapunov function. Thanks to Assumptions $\Hun$, $\Hdeux$, $\Htrois$ and $\Hquatre$, the following quantities are well defined:

\begin{align*}
\theta_1&\in \left( \max\left\{ \frac{\theta_*\lambda_{\max}}{p\theta_0\lambda_{\min}}; \frac{1}{p\theta_0\beta}\right\}~,~ \min\left\{p\theta_0,\bar{\theta}\right\} \right);\\
\alpha &\in \left( 0~,~ \min\left\{ \frac{p\theta_0\lambda_{\min}}{\theta_*}-\frac{\lambda_{\max}}{\theta_1} ; p\theta_0\beta\lambda_{\max}-\frac{\lambda_{\max}}{\theta_1} \right\}\right);\\
a &\in \left( 1+\frac{\lambda_{\max}}{\alpha\theta_1}~,~\min\left\{ \frac{p\theta_0\lambda_{\min}}{\alpha\theta_*}; \frac{p\theta_0\beta\lambda_{\max}}{\alpha}\right\}\right).
\end{align*}

We then consider the function $H$ defined on $E=\RR^d\times \mathcal{B}(1)$ by
\[ H(x,v)= \expo^{\alpha\normx} \left( a+ \varphi\left(\frac{x\cdot v}{\normx}\right) \right), \]
where $\varphi$ is a non-decreasing function of class $C^1$ on $[-1,1]$, with $\varphi(\theta)=\theta$ if $\theta\in [-1,-\theta_1]$, and $\varphi(\theta)=0$ if $\theta\in [0,1]$. We introduce $m$ the supremum norm of the derivative of $\varphi$, that is:
\[m=\sup_{\theta\in [-1,1]}| \varphi^{'}(\theta) |.\]
Then, we have the following result:

\begin{Prop}\label{Prop Lyapunov function}
There exist some constants $R>0$, $\eta>0$ such that for all $(x,v)\in (\RR^d\backslash\mathcal{B}(R))\times \mathcal{B}(1)$,
\[ LH(x,v)\leq -\eta H(x,v). \]
The function $H$ is then a Lyapunov function associated to the set $\mathcal{B}(R)\times \mathcal{B}(1)$ for the process $(X,V)$ driven by the generator $L$.
\end{Prop}

\begin{Rema}
This function $H$ can be compared to the one of \cite{BR,FGM15} in dimension $1$. It is also qualitatively equivalent to the ones of \cite{BRZ,DGM,DBCD}, but here it is written in terms of the jump rate $\lambda$, while in the papers cited before, they give a Lyapunov function depending on their potential $U$.
\end{Rema}

\begin{proof}
We can check that $H$ is in the domain of the generator $L$. For $(x,v)\in E$ we have:

\begin{equation}\label{LH2}
LH(x,v)= \expo^{\alpha \normx}\left( A_1 + A_2 + A_3 \right)
\end{equation}

with
\begin{align*}
A_1 &= \alpha\frac{x\cdot v}{\normx}\left( a + \varphi\left( \frac{x\cdot v}{\normx}\right) \right) \\
A_2 &= \varphi^{'}\left(\frac{x\cdot v}{\normx}\right)\frac{1}{\normx}\left( \normv^2-\left( \frac{x\cdot v}{\normx}\right)^2 \right) \\
A_3 &= \lambda(x,v)\int_{\mathcal{B}(1)}\left( \varphi\left( \frac{x\cdot v'}{\normx}\right) - \varphi\left( \frac{x\cdot v}{\normx} \right) \right)Q(x,v,\dt v').
\end{align*}

Let first remark that since the derivative of $\varphi$ is bounded by $m$, we always have 
\[ A_2 \leq \frac{m}{\normx}. \]
Moreover, thanks to Assumption $\Htrois$ and the definition of the function $\varphi$, if $x$ and $v$ are such that $\frac{x\cdot v}{\normx}>-\bar{\theta}$, we can bound $A_3$ as follows:
\begin{align*}
A_3 &\leq \lambda(x,v)\left[ \int_{\{ v'\in \mathcal{B}(1) / \frac{x\cdot v'}{\normx}\leq -\theta_0\}} \varphi\left( \frac{x\cdot v'}{\normx}\right) Q(x,v,\dt v') - \varphi\left( \frac{x\cdot v}{\normx}\right)\right]\\
&\leq \lambda(x,v) \left[ -\theta_0\int_{\{ v'\in \mathcal{B}(1) / \frac{x\cdot v'}{\normx}\leq -\theta_0\}} Q(x,v,\dt v') - \varphi\left( \frac{x\cdot v}{\normx}\right)\right]\\
&\leq -\lambda(x,v)\left[ p\theta_0+\varphi\left( \frac{x\cdot v}{\normx}\right) \right].
\end{align*}
We can now study Equation \eqref{LH2} depending on the different values taken by $\frac{x\cdot v}{\normx}$.
\begin{itemize}
\item[$\bullet$] If $\frac{x\cdot v}{\normx}\in [-1,-\theta_1]$:
\begin{align*}
LH(x,v)&\leq \expo^{\alpha\normx} \left[ -\alpha \theta_1(a-1)+ \frac{m }{\normx}+\lambda_{\max}\right].
\end{align*} 
\item[$\bullet$] If $\frac{x\cdot v}{\normx}\in (-\theta_1,0)$:
\begin{align*}
 LH(x,v)&\leq \expo^{\alpha\normx}\left[ \frac{m}{\normx}-\lambda(x,v)\left(p\theta_0-\theta_1\right)\right] \\
 &\leq \expo^{\alpha \normx}\left[ \frac{m}{\normx}-\lambda_{\min}\left(p\theta_0-\theta_1\right)\right].
 \end{align*}
\item[$\bullet$] If $\frac{x\cdot v}{\normx}\in [0,\theta_*]$:
\begin{align*}
 LH(x,v)&\leq \expo^{\alpha\normx}\left[\alpha\theta_* a+\frac{m}{\normx}-\lambda_{\min}p\theta_0 \right].
 \end{align*}
\item[$\bullet$] If $\frac{x\cdot v}{\normx}\in (\theta_*, 1]$, and if $\normx \geq \Delta$, with $\Delta$ defined in Assumption $\Hquatre$:
\begin{align*}
LH(x,v)&\leq \expo^{\alpha\normx}\left[ \alpha a+\frac{m}{\normx}-p\theta_0\beta\lambda_{\max}\right].
\end{align*}
\end{itemize}
Let us now take $R>r$ where
\[ r=\max \left\{ \frac{m}{\alpha\theta_1(a-1)-\lambda_{\max} }; \frac{m}{\lambda_{\min}\left(p\theta_0-\theta_1\right)}; \frac{m}{p\theta_0\lambda_{\min}-\alpha\theta_*a }; \Delta; \frac{m}{p\theta_0\beta\lambda_{\max}-\alpha a}\right\}. \]
Thanks to the assumptions made on each parameter, the constant $r$ is well defined and finite. Therefore, the previous calculations give the existence of a constant $\eta>0$ such that for all $(x,v)\in (\RR^d\backslash\mathcal{B}(R))\times \mathcal{B}(1)$:
\[ LH(x,v)\leq -\eta\expo^{\alpha\normx}. \]
Finally, the function $(y,w)\longmapsto a+\varphi\left( \frac{y\cdot w}{\lvert y \lvert}\right)$ being bounded from above by $a$, we get:
\[ LH(x,v)\leq -\eta aH(x,v), \]
which ends the proof.
\end{proof}

\subsection{Proof of Theorem \ref{Theoprincipal}}\label{Subsection Proof of the main result}

Since we have already found a Lyapunov function, we still have to review three points in order to use Theorem \ref{Theo Lyapunov function} : the irreducibility and the aperiodicity of the process, and the existence of petite sets. These are the object of the following proposition.

\begin{Prop}\label{Prop petite set, irreducibility, aperiodicity}
All compact sets of the form $K\times \mathcal{B}(1)$, with $K$ a compact set of $\RR^d$, are petite for the process $(X,V)$.\\
Moreover, the process is irreducible and aperiodic.
\end{Prop}

\begin{proof}
We do not detail the proof since we just have follow the ideas of the proof of Lemma 2 in \cite{DBCD}. It is indeed enough to show that for all $M>0$, for all $(x,v)\in \mathcal{B}(M)\times\mathcal{B}(1)$, for all positive bounded function $\varphi: E\rightarrow \RR$, there exist a constant $C>0$ and a compact set $A\subset\RR^d\times \mathcal{B}(1)$, both independent of $M$ and $\varphi$, such that:
\[ \int_0^\infty \expo^{-t}\EE_{(x,v)}\left[ \varphi(X_t,V_t) \right] \dt t \geq C \int \int_A \varphi(y,w) \dt y \dt w.\]
to deduce the first result of the proposition. Let us mention that the proof is based only on Assumptions $\Hun$, $\Hdeux$ and $\Htrois$.\\
Then, the irreducibility of the process follows easily, and we refer once more to \cite{DBCD}, Lemma 3, for the aperiodicity.
\end{proof}

\begin{Rema}
Let us mention different proofs of this result in some particular cases.\\
In dimension $1$, a proof has been given in \cite{BR} for the process with a velocity equal to $-1$ or $+1$, with the same main ideas that in \cite{DBCD} and \cite{Mon}.\\
For the process in dimension $2$, if the jump kernel has a density with respect to the Lebesgue measure which is bounded from below by a strictly positive constant, and if the jump rate is bounded, the previous proposition can be proved by geometric considerations, as in the proof of Lemma 4.5 in \cite{CV}.

\end{Rema}

\begin{proof}[Proof of Theorem \ref{Theoprincipal}] 
By Proposition \ref{Prop Lyapunov function} and \ref{Prop petite set, irreducibility, aperiodicity}, all conditions of \ref{Theo Lyapunov function} are satisfied, so that the stated result follows.
\end{proof}

\section{Exponential moments for the invariant measure}\label{Section Exponential moments for the invariant measure}

In this section, we want to study the existence of exponential moments for the unique invariant probabilty measure of the process $(X,V)$, that we denote by $\pi$, whose existence follows from Theorem \ref{Theoprincipal}. \\

Let first say that if the lengths of the consecutive excursions away from the origin were independent and identically distributed, the process $(X,V)$ would be a regenerative process, and we could then apply standard results on regenerative processes (see \cite{Asm} or \cite{Coco} for instance) to collect some information on the invariant measure of the process. This is for instance what is done in \cite{FGM10}, since the simple Zig-zag process is a regenerative process when we look at it at the successive hitting-times of the origin. However, the excursions away from the origin of the generalized Zig-zag process do not satisfy this assumptions of independence and identical distribution because of the dependence in speed between two consecutive excursions.\\
Nevertheless, the generalized Zig-zag process is a semi-regenerative process (see \cite{Asm} or \cite{Coco} for the definition). Indeed, let $Z_0=0$ and $(Z_n)_{n\geq 1}$ be the successive times at which the process $X$ enters the ball $\mathcal{B}(R)$ (with $R$ defined in Proposition \ref{Prop Lyapunov function}). The PDMP $(X,V)$ is a semi-regenerative process associated to the renewal process $((X_{Z_n},V_{Z_n}),Z_n)_{n\geq 0}$.\\

Before stating the result on regenerative process that we will use in this part, let us first briefly speak about Harris-recurrence and non-arithmetic process (see for instance \cite{AsmRen,Asm,Coco}). The Markov chain $Y$ with state space $F$ is called Harris-recurrent if there exists a measurable set $A\subset F$, $c>0$, $m\geq 1$, and a distribution $\varphi$ such that
\begin{enumerate}
\item for all $y\in F$, $\PP_y\left(\tau_A<\infty \right)=1$, where $\tau_A=\inf\{n\geq 0, Y_n\in A \}$;
\item for all $y\in A$, $\PP_y\left(Y_m \in \cdot \right) \geq c \varphi (\cdot)$.
\end{enumerate}
Then, we recall that the process $(Y_n,T_n)_{n\geq 0}$ is non-arithmetic if the laws of the variables $T_n-T_{n-1}$, $n\geq 1$ have a part which is absolutely continuous with respect to the Lebesgue measure (see \cite{Coco} for the definition of an arithmetic process).
\\

We have the following result concerning the convergence of semi-regenerative process, that we can found in \cite{AsmRen,Asm} or \cite{Coco}:

\begin{Theo}\label{thmconvSP}
Let $(\Psi_t)_{t\geq 0}$ be a semi-regenerative process associated to the renewal process $(Y,T)=(Y_n,T_n)_{n\geq 0}$.\\
We suppose that $(Y,T)$ is non-arithmetic and that $Y$ is Harris-recurrent, and let $\nu$ be an invariant measure for this chain.\\
Let $f: F\longrightarrow \RR$ be a measurable positive function  such that $(z,t)\longmapsto \EE\left[f(\Psi_t)\lvert Y_0=z\right]$ is bounded on $F\times [0,t]$ for all $t$. We define $g:F\times \RR_+\longrightarrow \RR$ by $g(z,t)=\EE\left[ f(\Psi_t)\mathbf{1}_{t<T_1} \lvert Y_0=z\right]$.\\
We suppose :
\begin{enumerate}
\item for $\nu$-almost all $z\in F$, $t\mapsto g(z,t)$ is almost everywhere continuous with respect to the Lebesgue measure;
\item there exists $\Delta>0$ such that
\[ \int_F \sum_{n\in\NN} \sup_{n\Delta\leq t < (n+1)\Delta} \lvert g(z,t)\lvert \nu(dz) < \infty. \]
\end{enumerate}
Then for $\nu$-almost all $y\in F$ :
\begin{equation}\label{eqconvSP}
 \EE\left[f(\Psi_t)\lvert Y_0=y\right] \underset{t\rightarrow\infty}{\longrightarrow} \dfrac{\int_F \EE\left[\int_0^{T_1} f(\Psi_s) ds \lvert Y_0=z\right] \nu(dz) }{\int_F \EE\left[T_1 \lvert Y_0=z\right] \nu(dz) }.
\end{equation}
\end{Theo}

Consequently, the previous theorem can be applied to our process, and it implies the existence of exponential moments for the invariant measure $\pi$ of the PDMP $(X,V)$. More precisely we have the following result:

\begin{Theo}\label{Theo exponential moments of the invariant measure}
Let $(X_t,V_t)_{t\geq 0}$ be the generalized Zig-zag process satisfying Assumptions $\Hun$, $\Hdeux$, $\Htrois$ and $\Hquatre$.\\
Let $\eta$ be as in Proposition \ref{Prop Lyapunov function}. Then for all $0<\beta<\eta$ and all $\gamma>0$, we have :
\[ \int_{E} \mathrm{e}^{\beta \lvert x \lvert + \gamma \lvert v\lvert }\pi(dx,dv) <\infty, \]
where we recall that $\pi$ is the unique invariant probability measure of the process $(X,V)$.
\end{Theo}

\begin{Rema}
This result ensures that at the equilibrium, the process is concentrated around the origin. From a modelling point of view, it means that in long time, the bacterium evolves around the chemo-attractant and does not go too far from it.
\end{Rema}

This theorem is a consequence of Proposition \ref{Prop Lyapunov function}, which implies the existence of exponential moments for the hitting time of the compact set $\mathcal{B}(R)$ associated to the Lyapunov function. Let us precise this fact:

\begin{Prop}\label{Proposition hitting time of B(1)}
Let us note $\tau_{\mathcal{B}(R)}=\inf\{ t\geq 0, X_t\in \mathcal{B}(R)\}$ the hitting time of $\mathcal{B}(R)$. For all $(x,v)\in E$,
\[ \EE_{x,v}\left[ \expo^{\eta \tau_{\mathcal{B}(R)}} \right]\leq H(x,v), \]
where $\eta$ and $H$ are defined in Proposition \ref{Prop Lyapunov function}.
\end{Prop}

\begin{proof}
In order to make the proof easier to read, we note $\tau$ instead of $\tau_{\mathcal{B}(R)}$.\\
For $(x,v)\in E$, Dynkin's formula gives:
\begin{align*}
\EE_{x,v}\left[ \expo^{\eta (t\wedge\tau)}\right] &\leq  \EE_{x,v}\left[ H\left( X_{t\wedge\tau},V_{t\wedge\tau} \right) \expo^{\eta(t\wedge\tau)}\right]\\
&= H(x,v) + \EE_{x,v}\left[ \int_0^{t\wedge\tau} (\eta + L)H\left( X_{s\wedge\tau},V_{s\wedge\tau} \right) \expo^{\eta(s\wedge\tau)} \dt s \right]\\
&\leq H(x,v),
\end{align*}
the last inequality resulting from Proposition \ref{Prop Lyapunov function}.\\
Then, when $t$ goes to infinity, the monotone convergence theorem gives:
\[ \EE_{x,v}\left[ \expo^{\eta \tau} \right]\leq H(x,v).\]
\end{proof}

\begin{proof}[Proof of Theorem \ref{Theo exponential moments of the invariant measure}]
Let first remark that the chain $(X_{Z_n},V_{Z_n})_{n\geq 0}$ is Harris-recurrent. Indeed, let $A=\mathcal{S}_R^{d-1}\times\mathcal{B}(1)$ be the state space of $(X_{Z_n},V_{Z_n})_{n\geq 0}$ ($\mathcal{S}_R^{d-1}$ denotes the sphere of $\RR^d$ with radius $R$). We obviously have, for all $(x,v)\in A$, $\PP\left( \tau_A<\infty\right) =1$. The second point in the definition of the Harris-recurrence can be proved with $\varphi$ the Lebesgue measure on $\RR^d$, using in particular the fact that $Z_m$ is almost-surely finite for each $m\geq 1$ thanks to Proposition \ref{Proposition hitting time of B(1)}.\\
Moreover, the chain $((X_{Z_n},V_{Z_n}),Z_n)_{n\geq 0}$ is non-arithmetic because the law of the times $Z_{n+1}-Z_n$ has a part which is absolutely continuous with respect to the Lebesgue measure.\\
Let us now introduce $f : (x,v)\in \RR^d\times \mathcal{B}(1) \longmapsto \mathrm{e}^{\beta\lvert x \lvert + \gamma\lvert v \lvert}$ for $0\leq \beta<\eta$ and $\gamma\geq 0$.\\
We first observe that $f$ is a positive measurable function, and that 
\[((x,v),t)\longrightarrow\EE\left[ f(X_t,V_t) \lvert (X_{Z_n},V_{Z_n})=(x,v)\right] \] is bounded on $\left(\mathcal{S}_R^{d-1}\times\mathcal{B}(1)\right)\times [0,t]$ for all $t$.\\ 
Let $\nu$ be the unique invariant probability measure of the chain $(X_{Z_n},V_{Z_n})_{n\geq0}$, which exists since the chain is positive Harris-recurrent. \\
Let us define the function $g$ on $\left(\mathcal{S}_R^{d-1}\times\mathcal{B}(1) \right)\times \RR_+$ by
\[ g\left((x,v),t\right)=\EE_{x,v}\left[ \mathrm{e}^{\beta\lvert X_t\lvert + \gamma \lvert V_t \lvert} \mathbf{1}_{t<Z_1} \right], \]
and let us check if $g$ satisfies the assumption of Theorem \ref{thmconvSP}:
\begin{enumerate}
\item For $\nu$-almost all $(x,v)\in \mathcal{S}_R^{d-1}\times\mathcal{B}(1)$, the function $t\longmapsto g((x,v),t)$ is almost everywhere continuous  with respect to the Lebesgue measure since it is right-continuous and has thus an at most countable set of discontinuities.
\item Let $\Delta>0$. According to Markov inequality we have 
\[\PP_{x,v}\left( Z_1>t\right)\leq \expo^{-\eta t}\EE_{x,v}\left[\expo^{\eta Z_1}\right].\]
Using Proposition \ref{Proposition hitting time of B(1)} and the fact that the Lyapunov function $H$ is uniformly bounded on $\mathcal{S}_R^{d-1}\times\mathcal{B}(1)$, we get the existence of a finite constant $M$ such that for all $(x,v)\in \mathcal{S}_R^{d-1}\times\mathcal{B}(1)$,
\[ \PP_{x,v}\left( Z_1 >t\right) \leq M\expo^{-\eta t}. \]\\
Moreover, let remark that if $(X_0,V_0)\in \mathcal{S}_R^{d-1}\times\mathcal{B}(1)$, since $\lvert X_0 \lvert = R$ and the velocity of the process lives in $\mathcal{B}(1)$, then for all $t\geq 0$, $\lvert X_t\lvert \leq R+t$.\\
Therefore we have
\begin{align*}
\int_{\mathcal{S}^{d-1}\times\mathcal{B}(1)} &\sum_{n\in\NN} \sup_{n\Delta\leq t < (n+1)\Delta}  \lvert g\left( (x,v),t \right) \lvert \nu(\dt x,\dt v)\\
&= \int_{\mathcal{S}^{d-1}\times\mathcal{B}(1)} \sum_{n\in\NN} \sup_{n\Delta\leq t < (n+1)\Delta} \EE_{x,v}\left[ \mathrm{e}^{\beta \lvert X_t \lvert + \gamma \lvert V_t \lvert}\mathbf{1}_{t<Z_1}\right] \nu(\dt x,\dt v) \\
&\leq \int_{\mathcal{S}^{d-1}\times\mathcal{B}(1)} \sum_{n\in\NN} \sup_{n\Delta\leq t < (n+1)\Delta} \mathrm{e}^{\beta(R+t)+\gamma} \PP_{x,v}\left( Z_1>t\right) \nu(\dt x,\dt v)\\
&\leq \int_{\mathcal{S}^{d-1}\times\mathcal{B}(1)} \sum_{n\in\NN} \sup_{n\Delta\leq t < (n+1)\Delta} \mathrm{e}^{\beta(R+ t) +\gamma} M\mathrm{e}^{-\eta t}\nu(\dt x,\dt v)\\
&\leq M\mathrm{e}^{\beta R+\gamma} \sum_{n\in\NN} \mathrm{e}^{\beta (n+1)\Delta - \eta n\Delta}\\
&= M\mathrm{e}^{\beta(R+\Delta) + \gamma}\sum_{n\in\NN} \left(\mathrm{e}^{\Delta(\beta-\eta)}\right)^n.
\end{align*}
This quantity is finite since $\beta<\eta$.
\end{enumerate}
The function $g$ satisfying all the required assumptions, we can apply Theorem \ref{thmconvSP}: for $\nu$-almost all $(x_0,v_0)\in \mathcal{S}_R^{d-1}\times\mathcal{B}(1)$ we have
\[ \EE_{x_0,v_0}\left[ \mathrm{e}^{\beta\lvert X_t\lvert +\gamma\lvert V_t\lvert}\right] \underset{t\rightarrow\infty}{\longrightarrow} \dfrac{\int_{\mathcal{S}_R^{d-1}\times\mathcal{B}(1)} \EE_{x,v}\left[ \int_0^{Z_1} \mathrm{e}^{\beta\lvert X_s\lvert +\gamma\lvert V_s\lvert}\dt s \right]\nu(\dt x,\dt v) }{\int_{\mathcal{S}_R^{d-1}\times\mathcal{B}(1)}\EE_{x,v}\left[Z_1\right]\nu(\dt x,\dt v)}. \]
The quantity on the right is clearly finite thanks to the previous calculations. Moreover, the ergodic theorem gives:
\[ \EE_{x_0,v_0}\left[ \mathrm{e}^{\beta\lvert X_t\lvert +\gamma\lvert V_t\lvert}\right] \underset{t\rightarrow\infty}{\longrightarrow} \int_{\RR^d\times\mathcal{B}(1)} \mathrm{e}^{\beta \lvert x \lvert +\gamma \lvert v \lvert} \pi(\dt x,\dt v). \]
We then deduce that
\[ \int_{\RR^d\times\mathcal{B}(1)} \mathrm{e}^{\beta \lvert x \lvert +\gamma \lvert v \lvert} \pi(\dt x,\dt v) = \dfrac{\int_{\mathcal{S}_R^{d-1}\times\mathcal{B}(1)} \EE_{x,v}\left[ \int_0^{Z_1} \mathrm{e}^{\beta\lvert X_s\lvert +\gamma\lvert V_s\lvert} \dt s\right]\nu(\dt x,\dt v) }{\int_{\mathcal{S}_R^{d-1}\times\mathcal{B}(1)}\EE_{x,v}\left[Z_1\right]\nu(\dt x,\dt v)}, \]
and we have proved the theorem.
\end{proof}

\section{The particular case of dimension $1$}\label{Section The particular case of dimension $1$}

In this section, we study PDMPs in dimension $1$, but in a different way, and under some different assumptions (see Section \ref{Subsection Comparison between the two studies} for a comparison between the two approaches in dimension $1$). This section is thus complementary to the previous study applied at the one-dimension.\\
We consider here the PDMP $((X_t,V_t))_{t\geq 0}$ with values in $\RR\times [-1,1]$ with infinitesimal generator  given by
\begin{equation}\label{generateur dim1}
Lh(x,v)=v\partial_x h(x,v)+\lambda(x,v)\int_{-1}^{1} \left( h(x,v')-h(x,v)\right)Q(x,v,\mathrm{d}v'),
\end{equation}
where $Q(x,v,\cdot)$ is a probability kernel on $[-1,1]$. This process is the one-dimensional version of the process driven by \eqref{generateur}.\\  

The framework of this section is the following:
 \begin{itemize}
 \item[($\mathcal{A}_1$):] $Q(x,v,\dt v')=q(v,v')(\nu(\dt v')+\mu(\dt v'))$ with $\nu$ a discrete measure and $\mu$ a restriction of the Lebesgue measure. We denote by $\mathcal{V}$ the support of $Q$ and suppose that there exists a constant $q_{\min}>0$ such that $q(v,v')\geq q_{\min}$ for all $v,v'\in \mathcal{V}$;
 \item[($\mathcal{A}_2$):] The process is symmetric: $\mathcal{V}$ is symmetric with respect to $0$ and $\lambda(x,v)=\lambda(-x,-v)$ for all $(x,v)\in\RR\times\mathcal{V}$;
 \item[($\mathcal{A}_3$):] There exists $0<\lambda_{\min}$ such that for all $(x,v)\in \RR\times\mathcal{V}$, $\lambda_{\min}\leq \lambda(x,v)$, and the quantity $\sup_{x\geq 0, v\leq 0}\lambda(x,v)$ is finite.
 \end{itemize}
 
The importance of the existence of $q_{min}$ in Assumption ($\mathcal{A}_1$) will not explicitly appear in the following since we are not going to give all the proofs in detail, but we note that it is useful in Proposition \ref{Prop petite set dim1} and Theorem \ref{Theorem exponential moments of the invariant measure dim 1}. \\
Assumption $(\mathcal{A}_2)$ implies that the process is spatially symmetric with respect to the origin, which will allow us to reduce the number of computations. Nevertheless, the following results are still available without this symmetry.\\
The fact that we suppose the kernel $Q$ to be independent of the position of the bacterium implies that the chain composed by the velocities at the jump times is a Markov chain with kernel $Q$. In particular, contrary to what is assumed in higher dimension, the attraction of the bacterium by the origin is not favoured by the jump kernel since it does not take care of the position of the bacterium  with respect to the origin.\\ 
We make thus an additional assumption $(\mathcal{A}_4)$ that takes into account this attraction to the origin, because of the presence of a chemo-attractant there. This assumption is the one that makes the process ergodic.
 
 \begin{multline*}
  (\mathcal{A}_4) : \exists I_* \subset\left[ 0,\underset{x\geq 0,v\in (0,1]}{\inf}\frac{\lambda(x,v)}{v}\right) \text{ an interval } I_*, \exists 0<J_*<1, \forall \alpha\in I_*, \forall v'\in \mathcal{V}, \\  \int_{-1}^1 G(\alpha,v) Q(v',\dt v)\leq J_*,
 \end{multline*}
where
\begin{equation}\label{expression G}
G(\alpha,v)=\frac{\underset{x\geq 0}{\sup}\lambda(x,v)}{\underset{x\geq 0}{\sup}\lambda(x,v)-\alpha v}\mathbf{1}_{v<0} + \frac{\underset{x\geq 0}{\inf}\lambda(x,v)}{\underset{x\geq 0}{\inf}\lambda(x,v)-\alpha v}\mathbf{1}_{v\geq 0}
\end{equation}
for $\alpha\geq 0$ and $v\in \mathcal{V}$. \\

Before stating the theorem that gives the exponential ergodicity of the process, we make some remarks on Assumption $(\mathcal{A}_4)$.

\begin{Rema}\label{Rem lien avec CRS}
Let us give a sufficient condition to satisfy Assumption $(\mathcal{A}_4)$ when the jump kernel $Q$ does not depend on the previous velocity. \\
Let us write $J_{v'}(\alpha)=\int_{-1}^1 G(\alpha,v) Q(v',\dt v)$.\\
For all $v'\in [-1,1]$, the function $J_{v'}$ is $C^1$ on $[0, \underset{x\geq 0,v\in ]0,1]}{\inf}\frac{\lambda(x,v)}{v})$, is convex and
\[J_{v'}(0)=1 \text{ , } J_{v'}'(0)=\int_{-1}^1 \left(\frac{v}{\underset{x\geq 0}{\sup}\lambda(x,v)}\mathbf{1}_{v<0} + \frac{v}{\underset{x\geq 0}{\inf}\lambda(x,v)}\mathbf{1}_{v\geq 0} \right)Q(v',\dt v) \]
\[ \text{ and } \underset{\alpha\rightarrow \underset{x\geq 0,v\in ]0,1]}{\inf}\frac{\lambda(x,v)}{v}}{\lim} J_{v'}(\alpha) = + \infty.\]
If we assume that $J_{v'}'(0)<0$, we get that there exists a unique $\hat{\alpha}\in (0,\underset{x\geq 0,v\in ]0,1]}{\inf}\frac{\lambda(x,v)}{v})$ such that $J(\hat{\alpha})=1$, and then there exists an interval $\hat{I}\subset (0,\alpha)$ such that for all $\alpha\in \hat{I}$, $J_{v'}(\alpha)$ is uniformly bounded by a constant strictly smaller than $1$ on the interval $\hat{I}$.\\
Consequently, for all $v'\in [-1,1]$, 
\[\exists \hat{I}\subset (0,\underset{x\geq 0,v\in (0,1]}{\inf}\frac{\lambda(x,v)}{v}), \exists 0<\hat{K}<1, \forall \alpha\in \hat{I},   \int_{-1}^1 G(\alpha,v) Q(v',\dt v)\leq \hat{K}.\]
In Assumption $(\mathcal{A}_4)$, we suppose in addition that the interval $\hat{I}$ and the constant $\hat{K}$ are uniform in $v'$. In particular, if $Q(v',\dt v)=Q(\dt v)$, the assumption $J'(0)<0$, that is
\[   \int_{-1}^1 \left(\frac{v}{\underset{x\geq 0}{\sup}\lambda(x,v)}\mathbf{1}_{v<0} + \frac{v}{\underset{x\geq 0}{\inf}\lambda(x,v)}\mathbf{1}_{v\geq 0} \right)Q(\dt v)<0 \]
is sufficient for Assumption $(\mathcal{A}_4)$ to be satisfied.\\
This assumption is the equivalent of Assumption $(H3)$ made by Calvez, Raoul and Schmeiser in \cite{CRS}. The probability study carried out in this section covers thus the framework of \cite{CRS}.

\end{Rema}

\begin{Rema}\label{Rem lien avec BR FGM}
Let see that in the case where the kernel $Q$ does not depend on the previous velocity and is symmetric in the sense that $q(v)=q(-v)$ for all $v\in\mathcal{V}$, Assumption $(\mathcal{A}_4)$ holds in the particular case where  $\underset{x\geq 0}{\inf}\lambda(x,v) > \underset{x\leq 0}{\sup}\lambda(x,v)$ for all $v\in [0,1]$, i.e. when the velocity tends to jump even more when the bacterium goes away from $0$ than when it moves towards the origin. This case is the one considered for instance in \cite{BR}, \cite{FGM10} and \cite{FGM15}, with a velocity taking its values in $\{-1,+1\}$\\
Let us prove this fact. Under this assumption on the jump rate, with the same notations as in the previous remark, and using the symmetries of the process we have:
\begin{align*}
J_{v'}'(0)&=\int_{-1}^1 \left(\frac{v}{\underset{x\geq 0}{\sup}\lambda(x,v)}\mathbf{1}_{v<0} + \frac{v}{\underset{x\geq 0}{\inf}\lambda(x,v)}\mathbf{1}_{v\geq 0} \right)q(v)\dt v\\
&= \int_{-1}^0 \frac{v}{\underset{x\geq 0}{\sup}\lambda(x,v)}q(v)\dt v + \int_0^1 \frac{v}{\underset{x\geq 0}{\inf}\lambda(x,v)}q(v)\dt v\\
&= -\int_0^1 \frac{v}{\underset{x\geq 0}{\sup}\lambda(x,-v)}q(v)\dt v + \int_0^1 \frac{v}{\underset{x\geq 0}{\inf}\lambda(x,v)}q(v)\dt v\\
&= -\int_0^1 \frac{v}{\underset{x\leq 0}{\sup}\lambda(x,v)}q(v)\dt v + \int_0^1 \frac{v}{\underset{x\geq 0}{\inf}\lambda(x,v)}q(v)\dt v\\
&<0.
\end{align*}
And the end of Remark \ref{Rem lien avec CRS} ensures that Assumption $(\mathcal{A}_4)$ is satisfied.
\end{Rema}

Let us now give the main theorem of this section.

\begin{Theo}\label{Theoprincipal dim1}
Let $(X_t,V_t)_{t\geq 0}$ be the PDMP on $\RR\times [-1,1]$ whose infinitesimal generator is given by
\[
Lh(x,v)=v\partial_x h(x,v)+\lambda(x,v)\int_{-1}^1 \left( h(x,v')-h(x,v)\right)Q(x,v,\mathrm{d}v'). \]
Under Assumptions $(\mathcal{A}_1)$, $(\mathcal{A}_2)$, $(\mathcal{A}_3)$ and $(\mathcal{A}_4)$, the process is exponentially ergodic.
\end{Theo}

\subsection{The hitting time of the origin}

As mentioned before, we are going to estimate the exponential moments of the hitting times of compact sets in order to prove Theorem \ref{Theoprincipal dim1}. We introduce two new notations.\\
We denote by $Z$ the first hitting time of $0$, i.e.
\begin{equation}\label{temps de retour en 0 Z}
Z= \inf \{ t>0, X_t=0\},
\end{equation}
and $S$ is the index of the first jump-time at which the position of the process has changed its sign: $S=\inf\{ n\geq 1, X_{T_n}X_0\leq 0 \}$.\\
Our goal in this section is to get information on the exponential moments of $Z$. For this purpose, we will first study the random variable $S$. Then, the inequality $Z\leq \sum_{i=1}^{S} \tau_i$ a.s. will allow us to come back to $Z$.

\begin{Prop}
For all $(x_0,v_0)\in \RR\times \mathcal{V} $ and all $\alpha$ such that $\lvert \alpha \lvert\in I_*$, we have, for all $n\geq 1$,
\[ \PP_{x_0,v_0}\left( S>n \right) \leq \kappa_{\alpha,x_0,v_0}  {J_*}^{n-1}, \]
with
\[ \kappa_{\alpha,x_0,v_0}= \mathrm{e}^{\lvert\alpha x_0\lvert}G\left( \lvert \alpha \lvert,v_0\right), \]
where $G$ and $J_*$ are defined by Assumption $(\mathcal{A}_4)$.
\end{Prop}

\begin{proof}
In order to make the proof easier to read, we distinguish the cases where $x_0$ and $v_0$ are positive or negative. We first look at the case $x_0>0$ and $v_0\in \mathcal{V}\cap [0,1]$, and the other cases are similar because of the symmetry of the process.\\ \\
Let $x_0>0$ and $v_0\in \mathcal{V}\cap [0,1]$.\\
We have $S=\inf\{n\geq 1, X_{T_n}\leq 0\}$, and $X_{T_n}=X_0+\sum_{i=0}^{n-1} V_i\tau_{i+1}=x_0+v_0\tau_1+\sum_{i=1}^{n-1} V_i\tau_{i+1}$.\\
 Let $\alpha\in [0, \underset{x\geq 0,v\in (0,1]}{\inf}\frac{\lambda(x,v)}{v})$.\\
Since $x_0$ is positive, on the event $\{S>n \}$, $X_{T_n}$ is also positive, and the sign of $\alpha$ implies:

\begin{align*}
\PP_{x_0,v_0} (S>n) &\leq \EE_{x_0,v_0} \left[ \expo^{\alpha X_{T_n}} \mathbf{1}_{S>n} \right]\\
 &= \expo^{\alpha x_0} \EE_{x_0,v_0} \left[ \expo^{\alpha\left( v_0\tau_1 +\sum_{i=1}^{n-1} V_i\tau_{i+1}\right)} \mathbf{1}_{S>n}\right]\\
 &=\expo^{\alpha x_0} \EE_{x_0,v_0} \left[ \expo^{\alpha v_0\tau_1}\prod_{i=1}^{n-1}\expo^{ V_i\tau_{i+1}} \mathbf{1}_{S>n}\right]\\
 &= \expo^{\alpha x_0} \EE_{x_0,v_0} \left[ \EE_{x_0,v_0}\left[\expo^{\alpha v_0\tau_1} \prod_{i=1}^{n-1}\expo^{ V_i\tau_{i+1}} \mathbf{1}_{S>n}\lvert X_{T_1},V_1,\ldots,X_{T_{n-1}},V_{n-1} \right]\right].
\end{align*}
Since the inter-jump times $\tau_1,\ldots,\tau_n$ are independent conditionally to the couples \\$(X_0,V_0),(X_{T_1},V_1),\ldots,(X_{T_{n-1}},V_{n-1})$, and since $\{S>n\}=\{X_{T_1}>0,\cdots,X_{T_n}>0\}$ we have:
\begin{equation}\label{P(S>n)}
\PP_{x_0,v_0} (S>n) \leq \expo^{\alpha x_0}\EE_{x_0,v_0}\left[\EE_{x_0,v_0}\left[\expo^{\alpha v_0\tau_1}\right] \prod_{i=1}^{n-1}\EE_{x_0,v_0}\left[ \expo^{\alpha V_i\tau_{i+1}} \lvert X_{T_i},V_i \right]\mathbf{1}_{X_{T_1>0},\cdots,X_{T_n>0}}\right].
\end{equation}
Moreover, we know that
\[ \PP\left(\tau_{i+1} \geq t \lvert X_{T_i}, V_i \right) = \overline{F}(t,X_{T_i},V_i), \]
which gives, for $1\leq i \leq n-1$:
\begin{align*}
\EE & \left[\expo^{\alpha V_i\tau_{i+1}}\lvert X_{T_i},V_i\right]\\
&= \int_0^{+\infty} \mathrm{e}^{\alpha V_i t} \lambda(X_{T_i}+V_i t, V_i)\mathrm{e}^{-\int_0^t \lambda(X_{T_i}+V_i s, V_i)\mathrm{d}s}\mathrm{d}t  \\
&= \int_0^{+\infty} \frac{\lambda(X_{T_i}+V_i t, V_i)}{\lambda(X_{T_i}+V_i t, V_i)-\alpha V_i}\left( \lambda(X_{T_i}+V_i t, V_i) - \alpha V_i \right) \mathrm{e}^{-\int_0^t \left( \lambda(X_{T_i}+V_i t, V_i) - \alpha V_i\right) \mathrm{d}s}\mathrm{d}t  .
\end{align*}
Furthermore, for $\alpha\in \RR$ and $v\in [-1,1]$ the function $\lambda\in\RR_+ \longmapsto\frac{\lambda}{\lambda-\alpha v}$ is increasing if $\alpha v\leq 0$ and decreasing otherwise. Therefore, if $V_i\geq 0$, since on $\{S>n\}$ the position of the bacterium is positive, we get:
\begin{align*}
\EE & \left[\expo^{\alpha V_i\tau_{i+1}} \lvert X_{T_i},V_i\right]\mathbf{1}_{X_{T_1>0},\cdots,X_{T_n>0}} \\
&\leq \frac{\underset{x\geq 0}{\inf}\lambda(x,V_i)}{\underset{x\geq 0}{\inf}\lambda(x,V_i)-\alpha V_i} \int_0^{+\infty} \left( \lambda(X_{T_i}+V_i t, V_i) - \alpha V_i \right) \mathrm{e}^{-\int_0^t \left( \lambda(X_{T_i}+V_i t, V_i) - \alpha V_i\right) \mathrm{d}s}\mathrm{d}t \\
&= \frac{\underset{x\geq 0}{\inf}\lambda(x,V_i)}{\underset{x\geq 0}{\inf}\lambda(x,V_i)-\alpha V_i}.
\end{align*}
In the case $V_i<0$ we get in the same way:
\begin{equation*}
\EE \left[\expo^{\alpha V_i\tau_{i+1}} \lvert X_{T_i},V_i\right] \mathbf{1}_{X_{T_1>0},\cdots,X_{T_n>0}}\leq \frac{\underset{x\geq 0}{\sup}\lambda(x,V_i)}{\underset{x\geq 0}{\sup}\lambda(x,V_i)-\alpha V_i}.
\end{equation*}
We have thus obtained the following inequality:
\[ \EE \left[\expo^{\alpha V_i\tau_{i+1}} \lvert X_{T_i},V_i,S\right]\mathbf{1}_{X_{T_1>0},\cdots,X_{T_n>0}} \leq G(\alpha,V_i), \]
where $G$ is given by \eqref{expression G}. \\
Getting back to \eqref{P(S>n)}, we get:
\[
\PP_{x_0,v_0}  \left(S\right.\left. >n\right) \leq \expo^{\alpha x_0}G(\alpha,v_0) \EE_{x_0,v_0}\left[ \prod_{i=1}^{n-1} G(\alpha,V_i) \right].
\]
Using now the fact that $(V_i)_{i\geq 0}$ is a Markov chain with kernel $Q$, we have:
\[
\EE_{x_0,v_0}\left[ \prod_{i=1}^{n-1} G(\alpha,V_i) \right]
 \leq \int_{-1}^1 G(\alpha,v_1)\times \ldots \times \int_{-1}^1 G(\alpha,v_{n-1}) Q(v_{n-2},\dt v_{n-1}) \ldots Q(v_0,\dt v_1).
\]
We deduce from Assumption $(\mathcal{A}_4)$, that if $\alpha\in I_*$ we have $\EE_{x_0,v_0}\left[ \prod_{i=1}^{n-1} G(\alpha,V_i) \right] \leq {J_*}^{n-1}$. And finally:
\begin{equation}\label{P(S>n)1}
\PP_{x_0,v_0}\left( S>n\right) \leq \expo^{\alpha x_0}G(\alpha,v_0) {J_*}^{n-1},
\end{equation}
for all $\alpha\in I_*$ , $x_0>0$ and $v_0\in\mathcal{V}\cap [0,1]$.\\ \\

For $x_0>0$ and $v_0\in\mathcal{V}\cap [-1,0)$, the calculations are exactly the same.\\
The case $x_0\leq 0$ can also be made in the same way. The symmetry of the process
 and Assumption $(\mathcal{A}_4)$ ensure that, for all $\alpha$ such that $-\alpha\in I_*$, for all $x_0\leq 0$ and all $v_0\in \mathcal{V}$,
\begin{equation}\label{P(S>n)2}
\PP_{x_0,v_0}\left( S>n\right) \leq \expo^{\alpha x_0}G(-\alpha,v_0){J_*}^{n-1}.
\end{equation}

Finally, we have proved in $\eqref{P(S>n)1}$ and $\eqref{P(S>n)2}$ that for all $(x_0,v_0)\in \RR\times \mathcal{V}$ and all $\alpha$ such that $\lvert \alpha\lvert \in I_*$,
\begin{equation}
\PP_{x_0,v_0}\left( S>n\right) \leq \expo^{\lvert \alpha x_0\lvert }G(\lvert \alpha\lvert,v_0) {J_*}^{n-1},
\end{equation}
which is the result of the proposition.

\end{proof}

\begin{Prop}\label{Prop temps de retour en 0}
For all $(x_0,v_0)\in \RR\times\mathcal{V}$, all $0<\rho < \lambda_{\min}(1-J_*)$ and all $\alpha$ such that $\lvert \alpha \lvert$ we have:
\[ \EE_{x_0,v_0}\left[ \expo^{\rho Z} \right] \leq K_{\alpha,x_0,v_0}\frac{1}{1-\frac{\lambda_{\min}}{\lambda_{\min}-\rho}J_*}, \]
where
\[ K_{\alpha,x_0,v_0}=\kappa_{\alpha,x_0,v_0}\frac{1-J_*}{{J_*}^2}=\mathrm{e}^{\lvert\alpha x_0\lvert}G\left( \lvert \alpha \lvert,v_0\right)\frac{1-J_*}{{J_*}^2}. \]
Moreover, $K_{\alpha,x_0,v_0}$ is uniformly bounded from above for $(x_0,v_0)\in [-x_1,x_1]\times\mathcal{V}$ for all $x_1>0$.
\end{Prop}

\begin{proof}
As mentioned before, we are going to use the inequality $Z\leq \sum_{i=1}^{S} \tau_i$ a.s. to get some informations on the hitting time of the origin.\\
Let us notice the variables $\tau_i, i\geq 1$ are stochastically smaller than i.i.d. exponential variables $E_i, i\geq 1$ with mean $\frac{1}{\lambda_{\min}}$. Indeed, the survival function of the variable $\tau_{i+1}$, conditionally to $X_{T_i}$ and $V_i$, is $\overline{F}(\cdot,X_{T_i},V_i)=\exp\left(\int_0^t \lambda(X_{T_i}+V_is,V_i)\dt s\right)$, and the jump rate $\lambda$ is uniformly bounded from below by $\lambda_{\min}$. The variables $(E_i)_{i\geq 1}$ can be taken as independent because of the independence of the variables $(\tau_i)_{i\geq 1}$ conditionally to $\left( X_{T_i},V_i \right)_{i\geq 0}$, and can also be taken as independent of the variables $(\tau_i)_{i\geq 1}$.\\
Thanks to this comment and the previous proposition, we get, for $\rho>0$ and $\alpha$ such that $\lvert \alpha \lvert\in I_*$:

\begin{align*}
\EE\left[ \mathrm{e}^{\rho Z }\right] &\leq \EE_{x_0,v_0}\left[ \mathrm{e}^{\rho \sum_{i=1}^{S} \tau_i}\right]\\
&\leq \EE_{x_0,v_0}\left[ \mathrm{e}^{\rho \sum_{i=1}^{S} \tilde{E}_i}\right] \\
&= \EE_{x_0,v_0} \left[ \EE\left[ \mathrm{e}^{\rho \sum_{i=1}^S \tilde{E}_i} \lvert S \right] \right]\\
&= \EE_{x_0,v_0} \left[\prod_{i=1}^S \EE\left[\mathrm{e}^{\rho \tilde{E}_i}\right]\right] \\
&=\EE_{x_0,v_0}\left[ \left(\frac{\lambda_{min}}{\lambda_{min}-\rho}\right)^S\right]\\ 
&= \sum_{n=0}^{+\infty}\left(\frac{\lambda_{min}}{\lambda_{min}-\rho}\right)^n \PP_{x_0,v_0}(S=n)\\
&= \sum_{n=0}^{+\infty}\left(\frac{\lambda_{min}}{\lambda_{min}-\rho}\right)^n \left(\PP(S_{x_0,v_0}>n-1)-\PP(S_{x_0,v_0}>n)\right)\\
&\leq  \expo^{\lvert \alpha x_0\lvert }G(\lvert \alpha\lvert,v_0)\frac{1-J_*}{J_*^2} \sum_{n=0}^{+\infty}\left(\frac{\lambda_{min}}{\lambda_{min}-\rho}J_*\right)^n .
\end{align*}
Finally, for all $\rho>0$ such that $\frac{\lambda_{min}}{\lambda_{min}-\rho}J_*<1$, we get :
\begin{equation}\label{momentsexpoT0}
\EE\left[ \mathrm{e}^{\rho Z }\right] \leq \expo^{\lvert \alpha x_0\lvert }G(\lvert \alpha\lvert,v_0)\frac{1-J_*}{J_*^2} \frac{1}{1-\frac{\lambda_{min}}{\lambda_{min}-\rho}J_*}<\infty,
\end{equation}
which ends the proof of the proposition.
\end{proof}

\subsection{Exponential ergodicity of the process}\label{Subsection Exponential ergodicity of the process}

In this section, we are going to give a proof of the exponential ergodicity of our process based on Proposition \ref{Prop temps de retour en 0}. For this purpose, we recall another result on the exponential ergodicity of a Markov process.

\begin{Theo}{(Theorem $6.2$ in \cite{DMT})}\label{Theo ergodicity tempsderetour}
Let $Y$ be an irreducible and aperiodic Markov process. Suppose that there exists a function $f\geq 1$, a closed measurable set $C\in E$ and some constants $\delta,\eta >0$, $M<\infty$ such that 
\[
\EE_x\left[ \int_0^{\tau_C(\delta)} \mathrm{e}^{\eta t}f(Y_t)\mathrm{d}t\right] < 
 \infty, \hspace{0.3cm} \text{for all } x\notin C 
\]
and 
\[
\sup_{x\in C}\EE_x\left[ \int_0^{\tau_C(\delta)} \mathrm{e}^{\eta t}f(Y_t)\mathrm{d}t\right] \leq M
\]

where $\tau_C(\delta)=\inf\{ t\geq \delta, Y_t\in C\}$.\\
If the set $C$ is petite for $Y$, then the process is exponentially ergodic.
\end{Theo}

\begin{Rema}
The proof of this theorem relies on the introduction of the function \\ $H_0(x):= 1 + \EE_x\left[ \int_0^{\tau_C(\delta)} \mathrm{e}^{\eta t}f(Y_t)\mathrm{d}t\right]$, which plays the role of a Lyapunov function for a Markov chain linked to the process $Y$.
\end{Rema}

To apply this theorem, we need then, as previously, to find a petite set. The following proposition states that all compact sets are petite for the process $(X,V)$, and is just a consequence of Proposition \ref{Prop petite set, irreducibility, aperiodicity}.

\begin{Prop}\label{Prop petite set dim1}
For all $x_1>0$, the set $C=[-x_1,x_1]\times \mathcal{V}$ is petite for the process $(X,V)$.
\end{Prop}

We can now prove Theorem \ref{Theoprincipal dim1}:

\begin{proof}[Proof of Theorem \ref{Theoprincipal dim1}]
Let $x_1>0$. Proposition \ref{Prop petite set dim1} ensures that the closed set $C=[-x_1,x_1]\times\mathcal{V}$ is petite for the process $(X,V)$.\\
Moreover, let us consider the quantity 
$\EE_{x,v}\left[ \mathrm{e}^{\eta\tau_C(\delta)}  \right]$ with $\eta,\delta >0$. It satisfies the assumptions of Theorem \ref{Theo ergodicity tempsderetour} (we have taken $f\equiv 1$) for $0<\eta < \lambda_{\min}(1-J_*)$ thanks to Proposition \ref{Prop temps de retour en 0}. \\
Finally, let $Leb$ be the Lebesgue measure on $\RR$. The process $(X,V)$ is $Leb\otimes Q$-irreducible, and aperiodic. The proof of these facts can be handled as mentioned previously in the general case.
The conclusion of Theorem \ref{Theo ergodicity tempsderetour} gives then the exponential ergodicity of $(X,V)$.
\end{proof}

\subsection{Exponential moments of the invariant measure}

As in the general case, we can show that the invariant measure, that we still denote by $\pi$, has exponential moments.

\begin{Theo}\label{Theorem exponential moments of the invariant measure dim 1}
Let $\gamma>0$ such that $\frac{\lambda_{\min}}{\lambda_{\min} - \gamma}J_*<1$. \\
For all $0<\alpha<\gamma$ and all $\beta>0$, we have :
\[ \int_{\RR\times [-1,1]} \mathrm{e}^{\alpha \lvert x \lvert + \beta \lvert v\lvert }\pi(dx,dv) <\infty. \]
\end{Theo}

\begin{proof}
The proof relies, as in Section \ref{Section Exponential moments for the invariant measure}, on Theorem \ref{thmconvSP}. We see the process $(X,V)$ as a semi-regenerative process between the successive hitting times of $0$. Then, using the upper bound of the exponential moments of the hitting time of $0$ obtained in Proposition \ref{Prop temps de retour en 0}, we can verify that the assumptions of Theorem \ref{thmconvSP} are satisfied, and the result follows.
\end{proof}

\subsection{Comparison between the two studies}\label{Subsection Comparison between the two studies}

Let us first recall that the proof of the exponential ergodicity carried out in any dimension is obviously relevant in dimension $1$, whereas the converse is not possible, or at least not directly. Indeed, the process in dimension $1$ is quite simple since it goes towards the origin or not in terms of the sign of the scalar product $X_t\cdot V_t$. In higher dimension, if we want to estimate the hitting time of a compact set, let say a ball, we can not see if the process is evolving towards the ball just in terms of the sign of $X_t\cdot V_t$, and the computations do not proceed as well as in dimension $1$.\\
Let now see that the two approaches carried out in dimension $1$ and in higher dimension cover different types of PDMPs.\\
\begin{itemize}
\item Let look at the process studied in \cite{FGM10}, \cite{FGM15} and \cite{BR}, whose generator has the following form :
\[ Lf(x,v) = v\partial_x f(x,v) + \lambda(x,v)\left( f(x,-v)-f(x,v)\right), \]
for $(x,v)\in\RR\times \{-1,+1\}$.\\
This process does not satisfy Assumption $(\mathcal{A}_4)$, because for $v'=-1$ we have 
\[\int_{-1}^1 G(\tilde{\alpha},v)Q(-1, \dt v)= G(\tilde{\alpha},1)= \frac{\underset{x\geq 0}{\inf}\lambda(x,1)}{\underset{x\geq 0}{\inf}\lambda(x,1)-\alpha }>1, \]
whereas it is still ergodic. Indeed, its ergodicity has been prooved in \cite{BR} or \cite{FGM15} for instance, with the introduction of a Lyapunov function, and this can also be deduce from Theorem \ref{Theoprincipal}, if the jump-rate $\lambda$ satisfies the good assumptions.
\item In the $d$-dimensional case, Assumption $\Hquatre$ assume a kind of monotony of the jump rate, which is not necessary in Assumption $(\mathcal{A}_4)$ in the one-dimensional case. We can thus construct a particular jump rate which, associated to the kernel $Q$, satisfies Assumption $(\mathcal{A}_4)$, but does not verify Assumption $\Hquatre$. Let for instance consider the case
\[ Lf(x,v) = v\partial_x f(x,v) + (2\mathbf{1}_{\mathrm{sgn}(x)v\geq -\frac{1}{2}}+\frac{1}{2}\mathbf{1}_{\mathrm{sgn}(x)v<-\frac{1}{2}})\frac{1}{2}\int_{-1}^1\left( f(x,v')-f(x,v)\right)\dt v', \]
for $(x,v)\in \RR\times [-1,1]$.\\
This process obviously does not satisfies Assumption $\Hquatre$ since for all $\theta_*\in [0,1]$, and all $\Delta>0$, $\inf_{\left\{ \mathrm{sgn}(x)v \geq \theta_*, \normx \geq \Delta \right\}}\lambda(x,v)=2=\sup_{\{\mathrm{sgn}(x)v \leq 0\}} \lambda (x,v)$.\\
Nevertheless, Assumption $(\mathcal{A}_4)$ is verified for this process. Indeed, let $\alpha\in \left( 0,\underset{x\geq 0,v\in (0,1]}{\inf}\frac{\lambda(x,v)}{v}\right) $, we have:
\begin{align*}
\frac{1}{2}&\int_{-1}^1 G(\alpha,v)\dt v\\
&= \frac{1}{2}\int_{-1}^{-\frac{1}{2}}\frac{\underset{x\geq 0}{\sup}\lambda(x,v)}{\underset{x\geq 0}{\sup}\lambda(x,v)-\alpha v} \dt v + \frac{1}{2}\int_{-\frac{1}{2}}^{-0}\frac{\underset{x\geq 0}{\sup}\lambda(x,v)}{\underset{x\geq 0}{\sup}\lambda(x,v)-\alpha v} \dt v + \frac{1}{2}\int_0^1 \frac{\underset{x\geq 0}{\inf}\lambda(x,v)}{\underset{x\geq 0}{\inf}\lambda(x,v)-\alpha v}\dt v\\
&= \frac{1}{2}\int_{-1}^{-\frac{1}{2}}\frac{\frac{1}{2}}{\frac{1}{2}-\alpha v}\dt v + \frac{1}{2}\int_{-\frac{1}{2}}^1 \frac{2}{2-\alpha v}\dt v\\
&= \frac{1}{4\alpha}\log \left(\frac{\frac{1}{2}+\alpha}{\frac{1}{2}+\frac{\alpha}{2}} \right) + \frac{1}{\alpha}\log \left( \frac{2+\frac{\alpha}{2}}{2-\alpha} \right) .
\end{align*}
A curve sketching shows that there exists an interval $ I_* \subset\left[ 0,\underset{x\geq 0,v\in (0,1]}{\inf}\frac{\lambda(x,v)}{v}\right)$ and a constant $J_*\in (0,1)$ such that for all $\alpha\in I_*$, $\frac{1}{2} \int_{-1}^1 G(\alpha,v)\dt v\leq J_*$, i.e. Assumption $(\mathcal{A}_4)$ is satisfied.
\end{itemize}

\paragraph{Acknowledgements.} The author wants to thank Alain Durmus, Arnaud Guillin and Pierre Monmarch\'e (\cite{DGM}) for fruitful discussions about the Lypaunov function introduced in the paper, and also H\'el\`ene Gu\'erin and Florent Malrieu for their help. This work was supported by the Agence Nationale de la Recherche project PIECE 12-JS01-0006-01.

%bibliographie
\bibliographystyle{plain}
\bibliography{biblio}

\begin{thebibliography}{10}

\bibitem{AsmRen}
Gerold Alsmeyer.
\newblock The markov renewal theorem and related results.
\newblock 01 1997.

\bibitem{ADNR}
C.~{Andrieu}, A.~{Durmus}, N.~{N{\"u}sken}, and J.~{Roussel}.
\newblock {Hypercoercivity of Piecewise Deterministic Markov Process-Monte
  Carlo}.
\newblock {\em ArXiv e-prints}, August 2018.

\bibitem{Asm}
S.~Asmussen.
\newblock {\em Applied Probability and Queues}.
\newblock Applications of mathematics : stochastic modelling and applied
  probability. Springer, 2003.

\bibitem{ABGKZ}
R.~Aza{\"i}s, J.-B. Bardet, A.~Genadot, N.~Krell, and P.-A. Zitt.
\newblock {Piecewise deterministic Markov process - recent results}.
\newblock {\em {ESAIM: Proceedings}}, 44:276--290, 2014.
\newblock Clermont-Ferrand, France, 29-31 August 2012.

\bibitem{BFR}
J.~{Bierkens}, P.~{Fearnhead}, and G.~{Roberts}.
\newblock {The Zig-Zag process and super-efficient sampling for bayesian
  analysis of big data}.
\newblock {\em ArXiv e-prints}, July 2016.

\bibitem{BR}
J.~Bierkens and G.~Roberts.
\newblock A piecewise deterministic scaling limit of lifted
  {M}etropolis--{H}astings in the {C}urie--{W}eiss model.
\newblock {\em Ann. Appl. Probab.}, 27(2):846--882, 2017.

\bibitem{BRZ}
J.~{Bierkens}, G.~{Roberts}, and P.-A. {Zitt}.
\newblock {Ergodicity of the zigzag process}.
\newblock {\em ArXiv e-prints}, December 2017.

\bibitem{BCDV}
A.~Bouchard-C\^{o}t\'{e}, S.~J. Vollmer, and A.~Doucet.
\newblock The {B}ouncy {P}article {S}ampler: a nonreversible rejection-free
  {M}arkov chain {M}onte {C}arlo method.
\newblock {\em J. Amer. Statist. Assoc.}, 113(522):855--867, 2018.

\bibitem{CRS}
V.~Calvez, G.~Raoul, and C.~Schmeiser.
\newblock Confinement by biased velocity jumps: aggregation of {\it
  {e}scherichia coli}.
\newblock {\em Kinet. Relat. Models}, 8(4):651--666, 2015.

\bibitem{CV}
N.~Champagnat and D.~Villemonais.
\newblock Exponential convergence to quasi-stationary distribution for absorbed
  one-dimensional diffusions with killing.
\newblock {\em ALEA Lat. Am. J. Probab. Math. Stat.}, 14(1):177--199, 2017.

\bibitem{Coco}
C.~Cocozza-Thivent.
\newblock {Processus de renouvellement markovien, Processus de Markov
  d{\'e}terministes par morceaux}.
\newblock {\em hal-01418366}, May 2018.

\bibitem{Dav}
M.~H.~A. Davis.
\newblock Piecewise-deterministic markov processes: a general class of
  non-diffusion stochastic models.
\newblock {\em Journal of the Royal Statistical Society. Series B
  (Methodological)}, 46(3), 1984.

\bibitem{DBCD}
G.~{Deligiannidis}, A.~{Bouchard-C{\^o}t{\'e}}, and A.~{Doucet}.
\newblock {Exponential ergodicity of the Bouncy Particle Sampler}.
\newblock {\em ArXiv e-prints}, May 2017.

\bibitem{DMT}
D.~Down, S.~P. Meyn, and R.~L. Tweedie.
\newblock Exponential and uniform ergodicity of {M}arkov processes.
\newblock {\em The Annals of Probability}, 23(4):1671--1691, 1995.

\bibitem{DGM}
A.~{Durmus}, A.~{Guillin}, and P.~{Monmarch{\'e}}.
\newblock {Geometric ergodicity of the bouncy particle sampler}.
\newblock {\em ArXiv e-prints}, July 2018.

\bibitem{FGM10}
J.~Fontbona, H.~Gu\'erin, and F.~Malrieu.
\newblock Quantitative estimates for the long-time behavior of an ergodic
  variant of the telegraph process.
\newblock {\em Adv. in Appl. Probab.}, 44(4):977--994, 12 2012.

\bibitem{FGM15}
J.~Fontbona, H.~Gu{\'e}rin, and F.~Malrieu.
\newblock {Long time behavior of telegraph processes under convex potentials}.
\newblock {\em {Stochastic Processes and their Applications}},
  126(10):3077--3101, 2016.

\bibitem{HHS}
C.-R. Hwang, S.-Y. Hwang-Ma, and S.-J. Sheu.
\newblock Accelerating diffusions.
\newblock {\em Ann. Appl. Probab.}, 15(2):1433--1444, 2005.

\bibitem{LNP}
T.~Leli\`evre, F.~Nier, and G.~A. Pavliotis.
\newblock Optimal non-reversible linear drift for the convergence to
  equilibrium of a diffusion.
\newblock {\em J. Stat. Phys.}, 152(2):237--274, 2013.

\bibitem{Lind}
T.~Lindvall.
\newblock {\em Lectures on the coupling method}.
\newblock Dover Books on Mathematics Series. Dover Publications, Incorporated,
  2002.

\bibitem{MT93a}
S.~P. Meyn and R.~L. Tweedie.
\newblock {\em Markov chains and stochastic stability}.
\newblock Springer-Verlarg, London, 1993.

\bibitem{MT93b}
S.~P. Meyn and R.~L. Tweedie.
\newblock Stability of markovian processes iii: Foster-{L}yapunov criteria for
  continuous-time processes.
\newblock {\em Advances in Applied Probability}, 25(3):518--548, 1993.

\bibitem{Mon}
P.~Monmarch\'e.
\newblock Piecewise deterministic simulated annealing.
\newblock {\em ALEA Lat. Am. J. Probab. Math. Stat.}, 13(1):357--398, 2016.

\bibitem{ODA}
H.~G. Othmer, S.~R. Dunbar, and W.~Alt.
\newblock Models of dispersal in biological systems.
\newblock {\em Journal of Mathematical Biology}, 26:263--298, 1988.

\bibitem{RW}
C.~{Wu} and C.~P. {Robert}.
\newblock {The Coordinate Sampler: a non-reversible Gibbs-like MCMC sampler}.
\newblock {\em ArXiv e-prints}, September 2018.

\end{thebibliography}

\end{document}